\documentclass[11pt]{amsart}

\usepackage[T1]{fontenc}

\usepackage{amsmath,amssymb,amsthm,array,color,multirow,pdflscape,graphicx,pigpen,stmaryrd,comment}
\usepackage[all]{xypic}
\usepackage[normalem]{ulem}

\usepackage{tikz}
\usepackage{tikz-cd}
\tikzset{shorten <>/.style={shorten >=#1,shorten <=#1}}

\usepackage{todonotes}

\usepackage[margin=1.5in]{geometry}

\usepackage[all,arc,2cell]{xy}
\UseAllTwocells

\DeclareMathAlphabet\mathbfcal{OMS}{cmsy}{b}{n}

\usepackage{xcolor}
\definecolor{darkgreen}{rgb}{0,0.30,0} 
\definecolor{darkred}{rgb}{0.75,0,0}
\definecolor{darkblue}{rgb}{0,0,0.6} 
\usepackage[pdfborder=0,pagebackref,colorlinks,citecolor=darkgreen,linkcolor=darkgreen,urlcolor=darkblue]{hyperref}
\renewcommand*{\backref}[1]{}
\renewcommand*{\backrefalt}[4]{({%
    \ifcase #1 Not cited.%
          \or On p.~#2%
          \else On pp.~#2%
    \fi%
    })}

\def\makeautorefname#1#2{\expandafter\def\csname#1autorefname\endcsname{#2}}
\makeautorefname{equation}{Equation}%
\makeautorefname{footnote}{footnote}%
\makeautorefname{item}{item}%
\makeautorefname{figure}{Figure}%
\makeautorefname{table}{Table}%
\makeautorefname{part}{Part}%
\makeautorefname{appendix}{Appendix}%
\makeautorefname{chapter}{Chapter}%
\makeautorefname{section}{Section}%
\makeautorefname{subsection}{Section}%
\makeautorefname{subsubsection}{Section}%
\makeautorefname{paragraph}{Paragraph}%
\makeautorefname{subparagraph}{Paragraph}%
\makeautorefname{theorem}{Theorem}%
\makeautorefname{thm}{Theorem}%
\makeautorefname{addm}{Addendum}%
\makeautorefname{mainthm}{Main theorem}%
\makeautorefname{corollary}{Corollary}%
\makeautorefname{cor}{Corollary}%
\makeautorefname{lemma}{Lemma}%
\makeautorefname{lem}{Lemma}%
\makeautorefname{sublemma}{Sublemma}%
\makeautorefname{sublem}{Sublemma}%
\makeautorefname{subl}{Sublemma}%
\makeautorefname{prop}{Proposition}%
\makeautorefname{property}{Property}
\makeautorefname{pro}{Property}
\makeautorefname{sch}{Scholium}%
\makeautorefname{step}{Step}%
\makeautorefname{conject}{Conjecture}%
\makeautorefname{conj}{Conjecture}%
\makeautorefname{questn}{Question}
\makeautorefname{quest}{Question}
\makeautorefname{qn}{Question}
\makeautorefname{definition}{Definition}%
\makeautorefname{defn}{Definition}%
\makeautorefname{defi}{Definition}%
\makeautorefname{def}{Definition}%
\makeautorefname{dfn}{Definition}%
\makeautorefname{df}{Definition}%
\makeautorefname{notation}{Notation}
\makeautorefname{notn}{Notation}
\makeautorefname{rem}{Remark}%
\makeautorefname{rems}{Remarks}%
\makeautorefname{rmk}{Remark}%
\makeautorefname{rk}{Remark}%
\makeautorefname{remarks}{Remarks}%
\makeautorefname{rems}{Remarks}%
\makeautorefname{rmks}{Remarks}%
\makeautorefname{rks}{Remarks}%
\makeautorefname{example}{Example}%
\makeautorefname{examp}{Example}%
\makeautorefname{exmp}{Example}%
\makeautorefname{exam}{Example}%
\makeautorefname{exa}{Example}%
\makeautorefname{axiom}{Axiom}%
\makeautorefname{axi}{Axiom}%
\makeautorefname{ax}{Axiom}%
\makeautorefname{case}{Case}%
\makeautorefname{claim}{Claim}%
\makeautorefname{clm}{Claim}%
\makeautorefname{assumpt}{Assumption}%
\makeautorefname{asses}{Assumptions}%
\makeautorefname{conclusion}{Conclusion}%
\makeautorefname{concl}{Conclusion}%
\makeautorefname{conc}{Conclusion}%
\makeautorefname{cond}{Condition}%
\makeautorefname{const}{Construction}%
\makeautorefname{con}{Construction}%
\makeautorefname{criterion}{Criterion}%
\makeautorefname{criter}{Criterion}%
\makeautorefname{crit}{Criterion}%
\makeautorefname{exercise}{Exercise}%
\makeautorefname{exer}{Exercise}%
\makeautorefname{exe}{Exercise}%
\makeautorefname{problem}{Problem}%
\makeautorefname{problm}{Problem}%
\makeautorefname{prob}{Problem}%
\makeautorefname{prob}{Problem}%
\makeautorefname{soln}{Solution}%
\makeautorefname{sol}{Solution}%
\makeautorefname{sum}{Summary}%
\makeautorefname{oper}{Operation}%
\makeautorefname{obs}{Observation}%
\makeautorefname{ob}{Observation}%
\makeautorefname{conv}{Convention}%
\makeautorefname{cvn}{Convention}%
\makeautorefname{warn}{Warning}%
\makeautorefname{note}{Note}%
\makeautorefname{fact}{Fact}%
\makeautorefname{ouch0}{Counterexample}%
\newtheorem{thm}{Theorem}[section]
\newtheorem{cor}{Corollary}[section]
\newtheorem{prop}{Proposition}[section]
\newtheorem{lem}{Lemma}[section]

\theoremstyle{definition}
\newtheorem{defn}{Definition}[section]
\newtheorem{df}{Definition}[section]
\newtheorem{con}{Construction}[section]

\newtheorem{exmp}{Example}[section]

\newtheorem{rem}{Remark}[section]

\newtheorem*{mainthm}{Main Theorem}

\newtheorem*{introconj}{Conjecture}

\makeatletter
\let\c@cor=\c@thm
\let\c@prop=\c@thm
\let\c@lem=\c@thm
\let\c@conj=\c@thm
\let\c@defn=\c@thm
\let\c@df=\c@thm
\let\c@exmp=\c@thm
\let\c@rem=\c@thm
\let\c@sch=\c@thm
\let\c@con=\c@thm
\let\c@conj=\c@thm
\let\c@notn=\c@thm
\let\c@equation\c@thm
\makeatother

\newcommand{\R}{\mathbb R}

\newcommand{\Z}{\mathbb Z}

\newcommand{\cat}[1]{\textup{\textbf{{#1}}}}

\newcommand{\Map}{\textup{Map}}

\newcommand{\id}{\textup{id}}
\newcommand{\colim}{\textup{colim}\,}
\newcommand{\hocolim}{\textup{hocolim}\,}

\newcommand{\ra}{\longrightarrow}

\newcommand{\sma}{\wedge}

\newcommand{\simar}{\overset\sim\longrightarrow}
\newcommand{\congar}{\overset\cong\longrightarrow}

\newcommand{\mc}{\mathcal}

\newcommand{\Mod}{\textup{Mod}}

\newcommand{\Wald}{\mathrm{Wald}}
\newcommand{\Seg}{\mathrm{Seg}}
\newcommand{\SMC}{\mathrm{SymMon}}
\newcommand{\cCMC}{\mathrm{coCartMon}}
\newcommand{\Sp}{\mathrm{Sp}}

\newcommand{\sM}{\mathcal{M}}
\newcommand{\sC}{\mathcal{C}}
\newcommand{\sR}{\mathcal{R}}

\newcommand{\sO}{\mathcal{O}}
\newcommand{\sF}{\mathcal{F}}

\newcommand{\GB}{ \mathcal{B}_G}
\newcommand{\GE}{\mathcal{GE}}

\newcommand{\bA}{\textbf{A}}

\newcommand{\sE}{\mathcal{E}}
\newcommand{\tG}{\mathcal{E}G}

\newcommand{\Rd}{R^{\delta}}
\newcommand{\Rhd}{R^{h\delta}}
\newcommand{\Rhf}{R^{hf}}

\newcommand{\po}{\ar@{}[dr]|{\text{\pigpenfont R}}}
\newcommand{\pb}{\ar@{}[dr]|{\text{\pigpenfont J}}}

\DeclareMathOperator{\Cat}{\mathrm{Cat}}

\newcommand{\sbt}{\,\begin{picture}(-1,1)(0.5,-1)\circle*{1.8}\end{picture}\hspace{.05cm}}

\newlength{\storeparskip}
\title[The equivariant parametrized $h$-cobordism theorem]{The equivariant parametrized $h$-cobordism theorem,  The non-manifold part}
\author{Cary Malkiewich}
\address{Department of Mathematics, Binghamton University}
\email{malkiewich@math.binghamton.edu}
\author{Mona Merling}
\address{Department of Mathematics, The University of Pennsylvania}
\email{mmerling@math.upenn.edu}

\makeatletter
\@namedef{subjclassname@2020}{\textup{2020} Mathematics Subject Classification}
\makeatother

\thanks{C. Malkiewich was partly supported by NSF grant DMS-2005524. M. Merling was partly supported by NSF grants DMS-1709461 and NSF CAREER grant DMS-1943925.}

\begin{document}

\begin{abstract}
%In \cite{CaryMona} we have introduced a construction of equivariant $A$-theory, whose fixed points satisfy a tom Dieck style splitting over conjugacy classes of subgroups. 

We construct a map from the suspension $G$-spectrum  $\Sigma_G^\infty M$ of a smooth compact $G$-manifold to the equivariant $A$-theory spectrum $\bA_G(M)$, and we show that its fiber is, on fixed points, a wedge of stable $h$-cobordism spectra. This map is constructed as a map of spectral Mackey functors, which is compatible with tom Dieck style splitting formulas on fixed points. In order to synthesize different  definitions of the suspension $G$-spectrum as a spectral Mackey functor,  we present a new perspective on spectral Mackey functors, viewing them as multifunctors on indexing categories for ``rings on many objects" and modules over such. This perspective should be of independent interest.

%$\bA_G$ satisfies an equivariant version of the stable parametrized $h$-cobordism theorem of Waldhausen for $G$-manifolds. 

\end{abstract}

\maketitle

%\setcounter{tocdepth}{2}
%\tableofcontents

\begingroup%
\setlength{\parskip}{\storeparskip}% Restore \parskip within this scope
\tableofcontents
\endgroup%

\section{Introduction}

Algebraic $K$-theory of spaces was introduced by Waldhausen in the 1970s in \cite{waldhausen1978alg}  to formulate a parametrized version of the classical $h$-cobordism theorem. For a closed smooth manifold $M$, the classical $h$-cobordism and $s$-cobordism theorems identify the connected components of the stable $h$-cobordism space $\mathcal{H}^\infty(M)$, using the  algebraic $K$-group $K_1(\Z[\pi_1 M])$. By contrast, the stable parametrized $h$-cobordism theorem relates the entire homotopy type of the space $\mathcal{H}^\infty(M)$ to a space arising from algebraic $K$-theory.

While the  proof of the stable parametrized $h$-cobordism theorem and the completion of Waldhausen's program was only accomplished 30 years later by Waldhausen, Jahren and Rognes in \cite{waldhausennew},  $A$-theory has firmly established itself since its introduction as an essential tool in the study of high-dimensional manifolds. 
It encodes the correct higher version of Whitehead torsion, codifying the difference between the block diffeomorphism group constructed by surgery theory and the actual diffeomorphism group, not just on $\pi_0$ but on the higher homotopy groups as well \cite{ww}. Moreover, if the manifolds are allowed to have boundary and are stabilized with respect to arbitrary vector bundles, then $A$-theory (or more precisely the difference between $\Sigma^\infty_+ X$ and $\bA(X)$) captures the entire difference between a compact smooth manifold and its underlying homotopy type \cite{waldhausen_manifolds}; see for instance the main result of \cite{dww}.

When a manifold $M$ has a group action by a finite group $G$, we have constructed an equivariant $A$-theory spectrum $\bA_G(M)$ with the property that on $H$-fixed points for every $H$, it satisfies a tom Dieck formula
$$\bA_G(M)^H\simeq  \prod_{(K) \leq H} \bA (M^K_{hWK}).$$ Here $WK$ is the Weyl group $N_HK/K$, and $(H) \leq G$ denotes conjugacy classes of subgroups. On $\pi_1$, it therefore contains as a summand the classical $H$-equivariant Whitehead group. It has been known since the 90s that this fixed point decomposition is necessary for any definition of an equivariant $A$-theory spectrum that hopes to be relevant to equivariant $h$-cobordisms and pseudoisotopies. For a different application, a generalization of the Segal conjecture, Rognes has laid some of the groundwork for this construction in unpublished work from that time.

In \cite{CaryMona}, we have used the newly developed technology of spectral Mackey functors  \cite{ GM1, BO, BO, Gmonster, Gmonster2} to construct the $G$-spectrum $\bA_G(M)$ with the desired fixed points for all subgroups. There is a na{\"i}ve guess as the first definition of equivariant $A$-theory, corresponding to the equivariant $K$-theory in \cite{mona_thesis} and the $K$-theory of group actions in \cite{Gmonster2}. In \cite{CaryMona} we call it $\bA_G^\textup{coarse}(X)$ in  because on fixed points it only sees ``coarse" $G$-equivalences, namely $G$-maps that are nonequivariant weak homotopy equivalences. On fixed points it recovers the bivariant $A$-theory of the fibration $EG \times_H X \ra BH$ as defined by Williams, see \cite{bruce, malkiewich_merling_coassembly}. However, $\bA_G^\textup{coarse}(X)$ does not match our expected input for the $h$-cobordism theorem, and this paper will focus solely on $\bA_G(X)$ for a $G$-space $X$.

The expected connection to equivariant $h$-cobordisms can be made more precise in the following conjecture, which is analogous to the non-equivariant result, and which is inspired by Goodwillie's vision for equivariant $A$-theory.

\begin{introconj}
The $G$-spectrum  $\bA_G(X)$ defined in \cite{CaryMona} satisfies a splitting \begin{equation*}\bA_G(X)\simeq \Sigma^\infty_G X \times \textbf{Wh}_G(X),\end{equation*} where $\Omega  \textbf{Wh}_G(X)$ is a $G$-spectrum whose  zeroth space is the space  of equivariant  $h$-cobordisms $\mathcal{H}_G^\infty(X)$, stabilized with respect to representation disks $D(V)$.
\end{introconj}

In this paper we take the first step in the proof of this conjecture.  We construct a map $\Sigma_G^\infty X_+\to \bA_G(X)$ for every $G$-space $X$, and in the case of a compact smooth $G$-manifold, we show that the fixed points of the fiber of this map split up as a product over stable $h$-cobordism spaces. We summarize the content of \autoref{inclusions}, \autoref{all_natural_in_X},  \autoref{it_models_suspension_spectrum} and \autoref{fiber_identification} as follows.

\begin{mainthm}
Let $M$ be a smooth compact $G$-manifold. There is  a fiber sequence of $G$-spectra $$\textbf{H}_G(M)\longrightarrow \Sigma_G^\infty M_+\longrightarrow \bA_G(M)$$ where $\textbf{H}_G(M)$ is a $G$-spectrum with the property that its $H$-fixed points have underlying infinite loop space
$$\Omega^\infty \textbf{H}_G(M)^H\simeq  \prod_{(K) \leq H} \mathcal{H}^\infty (M^K_{hWK}),$$
 where $\mathcal{H}^\infty$   denote the stable $h$-cobordism spaces.
\end{mainthm}

Part of the content of the theorem above is to construct the map $\Sigma_G^\infty X_+\to \bA_G(X)$ appropriately. The $G$-spectrum $\textbf{H}_G(M)$ is then defined as the fiber of this map, so that a priori it has nothing to do with $h$-cobordism spaces. However, on the $G$-fixed points we can apply the nonequivariant stable parametrized $h$-cobordism theorem, and conclude that its infinite loop space splits over conjugacy classes into nonequivariant stable $h$-cobordism spaces $\mathcal{H}^\infty(M^H_{hWH})$. We can then describe each of these spaces as the space of $h$-cobordisms on the subspace of $M$ whose isotropy is exactly $H$, with $M$ stabilized by representation discs $D(V)$. We call this  ``the non-manifold part" of the equivariant stable $h$-cobordism theorem as a homage to the classical proof of \cite{waldhausennew}, though we note that it is not the equivariant analogue of that part of their proof.

The ``manifold part" will be the further step of showing that this product space can be re-assembled into a single space of $G$-equivariant $h$-cobordisms on $M$, stabilized with respect to representation disks. This requires significantly more geometric techniques and will appear in joint work with Goodwillie and Igusa. The necessary definitions and properties of $h$-cobordism spaces of $G$-manifolds are also needed in current work of Igusa and Goodwillie on equivariant higher torsion invariants. 

%Together with Goodwillie and Igusa, we plan to prove that a space of equivariant $h$-cobordisms, stabilized with respect to representation disks $D(V)$, satisfies a ``tom Dieck style" splitting 
%%\begin{equation}\label{splitting}
%$$\mathcal{H}_G^\infty(X)\simeq \prod_{(H)} \mathcal{H}^\infty(X^H_{hWH})$$
%%\end{equation}
%That result will identify the fiber on fixed points with $G$-spectrum whose fixed points have as underlying infinite loop spaces, stable  spaces  of equivariant $h$-cobordisms, stabilized with respect to representation disks. The necessary definitions and properties of $h$-cobordism spaces of $G$-manifolds are also needed in current work of Igusa and Goodwillie on equivariant higher torsion invariants. 

The idea of the proof of the main theorem is to construct a  map from the equivariant suspension spectrum to the equivariant $A$-theory spectrum which is compatible with the splittings over conjugacy classes on fixed points. We adopt the same ``spectral Mackey functor'' approach we took in \cite{CaryMona}, so the strategy is to glue together these inclusions along restriction and transfer maps into a map of $G$-spectra. This is harder than it appears because it requires us to synthesize different definitions of the suspension $G$-spectrum as a spectral Mackey functor in order to reconcile it with the formulas in \cite{BD}. In order to achieve this, we  give  a new perspective on spectral Mackey functors in terms of multicategories and multifunctors as a way to express ``rings on many objects and modules over such,"  which makes the analysis more streamlined.
 
The paper is organized as follows. In \autoref{spectralmack} we review spectral Mackey functors and reinterpret them as multifunctors from parameter multicategories for rings on many objects and modules over such. % This uses the ideas of the beautiful paper \cite{EM} and generalizes the language to encode not just rings, but rings on many objects in multicategories.
Using this language and the multiplicative comparison of Segal and Waldhausen $K$-theory from \cite{bohmann_osorno}, we prove an equivalence between two distinct notions of spectral Mackey functor from \cite{GM1} and from \cite{CaryMona}.
%it becomes immediate to reconcile spectral Mackey functors defined as modules over the spectral Burnside category $\GB^\Wald$ defined using $K$-theory of Waldhausen categories as mapping spectra and modules over the spectral Burnside category  $\GB$ defined using $K$-theory of permutative categories as mapping spectra.
In \autoref{map},  we construct a map of $G$-spectra $\Sigma_G^\infty X_+\rightarrow \bA_G(X)$ that respects the tom Dieck splittings on the fixed points. For this we use use the categories of $H$-equivariant ``homotopy discrete retractive spaces'' from \cite{BD}
%as a model for the $H$-fixed points of the suspension spectrum and the work laid out in the previous section allowing us to compare definitions of the suspension $G$-spectrum as modules over $\GB$ and $\GB^\Wald$.
and the comparison of Mackey functors from the previous section. Lastly, in  \autoref{fiber} we analyze the fiber of the map $\Sigma_G^\infty M_+\to \bA_G(M)$. Using the non-equivariant stable parametrized $h$-cobordism theorem, we show that it has a tom Dieck style decomposition as a product of nonequivariant stable $h$-cobordism spectra.

% \cite{AK}

%\subsection*{Conventions} All spaces are compactly generated weak Hausdorff, and all spectra are nonequivariant symmetric spectra, except in a few highlighted places where we use orthogonal $G$-spectra.
%Our work here focuses on the category of smooth compact $G$-manifolds. 
%We point out that even though it is reasonable to expect similar results for $PL$-manifolds, we really cannot expect the same results in  the category of compact topological $G$-manifolds. 
%%% MONA: Maybe this comment is better in the next paper where we prove the splitting of the smooth h-cob space and we can give the counterexample of how that fails in the TOP case.

%although one might have to be careful about what exactly a $G$-manifold means.

\subsection*{Acknowledgements} We are especially thankful to Tom Goodwillie for sharing his vision and ideas that are at the core of this program. We also thank Clark Barwick, Wolfgang L{\"u}ck,   Stefan Schwede, and Shmuel Weinberger for helpful discussions and encouragement, and to Anna Marie Bohmann and Ang{\'e}lica Osorno for generously sharing drafts of their work in progress on the multiplicative comparison of Waldhausen and Segal $K$-theory before it appeared. We thank Wojtek Dorabia\l{}a and Bernard Badzioch for patiently explaining to us some of the subtleties in their work on the fixed points of equivariant $A$-theory, which we crucially use in this paper. Lastly, we are immensely grateful to the Max Planck Institute in Bonn for their hospitality during the fall 2018 semester when both authors were visiting the MPI.

\section{Spectral Mackey functors as multifunctors}\label{spectralmack}

We start by reviewing spectral Mackey functors, and giving them a different interpretation using the language of multifunctors. Recall that a \emph{spectrally enriched category} or \emph{spectral category} $\sR$ is a ``ring on many objects'' in $\Sp$: precisely, it is a collection of spectra $\sR_{a,b}$, one for each pair of objects $a,b$ in $\sR$, together with unit maps $S\to \sR_{a,a}$ and unital and associative composition maps $\sR_{a,b}\sma\sR_{c,a}\to \sR_{c,b}$. A spectrally enriched functor $\sR\to \Sp$ is ``module over $\sR$'': precisely, it is given by a spectrum $\sM_a$ for every object $a$ of $\sR$, and unital and associative action maps of spectra  $\sM_b \sma \sR_{a,b}\to   \sM_a$. For a  discussion of this perspective, which has inspired our point of view, see \cite{schwede_shipley_morita}. In this paper, our ring and module spectra all take values in the category of symmetric spectra from \cite{diagram,MMSS}, though the discussions would apply just as well with orthogonal spectra.

%We start  by making  precise  the meaning of  ``spectral Mackey functor." To do this, 

Let $G$ be a finite group. There is a ring $\GB$ on object set $\{G/H\}_{H\leq G}$ whose spectra $\GB(G/H, G/K)$ are formed by applying Segal $K$-theory to the symmetric  monoidal category of spans from $G/H$ to $G/K$. This spectral category was introduced in \cite{GM1}, where it is shown that modules over this ring are equivalent to $G$-spectra.

\begin{thm}\cite{GM1}
There is a Quillen equivalence $$\GB-\Mod\simeq G\Sp,$$ where $G\Sp$ is the category of orthogonal $G$-spectra.
\end{thm}

In \cite{CaryMona} we considered an analogous ring on many objects $\GB^\Wald$, defined instead as the Waldhausen $K$-theory of the categories of retractive spaces over $G/H\times G/K$. We claimed that forthcoming work would show that the rings $\GB$ and $\GB^\Wald$ are equivalent, and therefore their categories of modules are Quillen equivalent. In this first section we show how to make that theorem precise (\autoref{GB_equivalence}), by combining the main result of \cite{bohmann_osorno} with a framework for studying spectral categories that re-interprets modules and spectral enrichments in terms of maps of multifunctors.

This establishes that modules over $\GB$ and over $\GB^\Wald$ are equivalent to each other. We can therefore use the term ``spectral Mackey functor" to refer to a module over either of these two spectral categories. Along the equivalence to orthogonal $G$-spectra, every spectral Mackey functor produces an orthogonal $G$-spectrum. This allows us to build $G$-spectra in a categorical way from either symmetric monoidal or Waldhausen categories. In particular, the equivariant $A$-theory spectrum $\bA_G(X)$ we are interested in is defined as a module over $\GB^\Wald$.

%We recall the construction of two categories enriched in spectra, $\GB$ and $\GB^{\Wald}$. We show here, using a result of Bohmann and Osorno \cite{bohmann_osorno}, that the two are equivalent as spectral categories.
%pointwise equivalent. (They have the same objects, and equivalent mapping spectra.) 
%Therefore their categories of modules are Quillen equivalent \cite[6.1]{schwede_shipley_equivalences}. By \cite{GM1} they are therefore also Quillen equivalent to the model category of orthogonal $G$-spectra from \cite{MM}.
%
%In this paper, a spectral Mackey functor is a module over either of the two spectral categories $\GB$ or $\GB^\Wald$. Along the equivalence to orthogonal $G$-spectra, every spectral Mackey functor produces an orthogonal $G$-spectrum. This allows us to build $G$-spectra in a categorical way.

\subsection{Multicategories and multifunctors}\label{multisec} To describe $\GB$ and $\GB^{\Wald}$ and their comparison, and more importantly, to facilitate comparing modules over them, we will use the language of multicategories. Recall that a \emph{(non-symmetric) multicategory} $\cat C$ has a collection of objects, a set of $n$-ary morphisms $\cat C(a_1,\ldots,a_n;b)$ for every $n \geq 0$ and list of objects $a_1,\ldots,a_n,b$, composition maps
\[ \cat C(a_1^1,\ldots,a_1^{k_1};b_1) \times \ldots \times \cat C(a_n^1,\ldots,a_n^{k_n};b_n) \times \cat C(b_1,\ldots,b_n;c) \longrightarrow \cat C(a_1^1,\ldots,a_n^{k_n};c) \]
that are associative, and identity morphisms in $\cat C(a;a)$ that serve as left and right units for the composition. A \emph{multifunctor} $\cat C \to \cat D$ is a function of objects and of $n$-ary morphisms for every $n \geq 0$ that respects composition and identity morphisms. A \emph{multinatural transformation} $\eta$ of multifunctors $F \Rightarrow G$ assigns to each object $a$ of $\cat C$ a 1-ary morphism $\eta(a)\colon F(a) \to G(a)$, such that for every $n$-ary morphism $f$ in $\cat C$ the square formed from $F(f)$, $G(f)$, and various instances of $\eta$ commutes. A detailed treatment is given in \cite{yau}.

\begin{exmp}\label{multicatex}\hfill 
	\vspace{-1em}
	
	\begin{enumerate}
		\item Every monoidal category $\mc V$ gives a multicategory in which the $n$-ary morphisms are the ordinary morphisms from  $a_1 \otimes \ldots \otimes a_n$, with a fixed way to associate $\otimes$, to $b$. In particular, symmetric and orthogonal spectra form multicategories. A multifunctor between two multicategories arising from monoidal categories corresponds to a monoidal functor. 
		\item Any (non-symmetric) operad $\sO$ is a (non-symmetric) multicategory with one object. A multifunctor from $\sO$ to a multicategory that arises from a monoidal category as in the previous example corresponds precisely to an $\sO$-algebra structure on the object that the unique object of the multicategory $\sO$ maps to.
		\item The multicategory $\Wald$ has objects Waldhausen categories and $n$-ary morphisms multiexact functors. The equivalent multicategory $\Wald_\vee$ has the same morphisms but the objects are also equipped with a choice of direct sum functor.
		\item The multicategory $\SMC$ has objects the symmetric monoidal categories with strict unit, %\cmar{this doesn't seem necessary to get a multicategory, BO must use it for their main result somehow though},
		and $n$-ary morphisms the ``multilinear'' functors, see \cite{bohmann_osorno}. The multicategory $\SMC_{we}$ is defined in the same way except that the symmetric monoidal categories have chosen subcategories of weak equivalences, preserved by the tensor product.
		\item We can restrict the previous example further to cocartesian monoidal categories (with strict unit), giving a multicategory $\cCMC$.
	\end{enumerate}
	\end{exmp}
	
	\begin{rem}
		Note that \cite{EM} and \cite{BO, bohmann_osorno} work with \emph{symmetric} multicategories and multifunctors. This is a stronger notion that also encodes \emph{commutative} multiplicative structure. In this paper we only need to use non-symmetric multicategories and multifunctors, and we restrict our attention to them for simplicity.
	\end{rem}

	\subsection{Modules over rings on many objects as multifunctors}
	We can make sense of rings, modules, and other algebraic structures in a category that has a multicategory structure, and these notions correspond to the usual notions if the multicategory structure comes from a monoidal structure such as in \autoref{multicatex}(1). The first to introduce the idea of using multicategories to encode multiplicative structure at the categorical level and carry it over to $K$-theory spectra via multifunctors was \cite{EM}. 
	
	The idea employed  in \cite{EM} is to use small parameter multicategories that encode the desired algebraic structure and then ask for a multifunctor from such a parameter multicategory to any multicategory. For example, a multifunctor out of the multicategory with one object and all multimorphism sets given by $\ast$ to any multicategory $\mc M$ picks out an object in $\mc M$ that carries the structure of a ring. In the case when the multicategory structure on $\mc M$ comes from a  monoidal category structure, a ring in this sense is just a monoid with respect to the monoidal structure. We refer the reader to \cite{EM} for a thorough discussion and definitions for parameter multicategories encoding many kinds of algebraic structures. In the next definition we introduce parameter multicategories that encode ``rings on many objects" and modules over such, generalizing the definitions os parameter multicategories for rings and modules from \cite{EM}.
	
	\begin{defn} Let $S$ be a set. We define the following  parameter multicategories.
	\begin{enumerate}
		\item The multicategory $ \cat R_S$ has objects $S \times S$, 0-ary morphism sets
		\[ \cat R_S(;(s,t)) = \left\{ \begin{array}{ccl}
		{*} && s = t \\
		\emptyset && \textup{otherwise}
		\end{array} \right. \]
		and for $n > 0$, $n$-ary morphism sets
		\[ \cat R_S((s_1,t_1),\ldots,(s_n,t_n);(s,t)) = \left\{ \begin{array}{ccl}
		{*} && s = s_1, t_1 = s_2, \ldots, t_{n-1} = s_n, t_n = t \\
		\emptyset && \textup{otherwise}.
		\end{array} \right. \]
		When $S = *$, this multicategory is just the non-symmetric associative operad.
		\item The multicategory $\cat M_S$ has objects $(S \times S) \amalg S$, morphism sets identical to $\cat R_S$ when all the objects are in $(S \times S)$, and for $n \geq 0$, $n$-ary morphism sets that are $*$ for objects of the form
		\[ s_1,(s_1,s_2),(s_2,s_3),\ldots,(s_{n-1},s_n);s_n \]
		and $\emptyset$ otherwise.
		
	\end{enumerate}
\end{defn}

The definitions of $\cat R_S$ and $\cat M_S$ are arranged precisely to make \autoref{multi1}, \autoref{multi2}, and \autoref{multi3} true.
\begin{lem}\label{multi1}
	If $\mc V$ is a monoidal category then the category of multifunctors $\cat R_S \to \mc V$ and multinatural transformations is isomorphic to the category of $\mc V$-enriched categories with object set $S$ and $\mc V$-enriched functors. Moreover, the category of multifunctors $\cat M_S \to \mc V$ and multinatural transformations is isomorphic to the category  of $\mc V$-enriched modules and $\mc V$-enriched maps of modules.
\end{lem}
%\begin{lem}\label{multi1}
%	If $\mc V$ is a monoidal category then each multifunctor $\cat R_S \to \mc V$ corresponds to a $\mc V$-enriched category with object set $S$. Each multifunctor $\cat M_S \to \mc V$ corresponds to a $\mc V$-enriched module over that category. Natural transformations of such multifunctors give maps of enriched categories and maps of modules under them.
%\end{lem}

\begin{exmp}
	If $\mc V=\Sp$ is the category of orthogonal or symmetric spectra, then a multifunctor $\cat R_S\rightarrow \Sp$ corresponds to a spectral category $\sR$ with object set $S$, which can be viewed as a ring ``on many objects" in spectra. %For a discussion of this perspective, see \cite{schwede_shipley_morita}. %The mapping spectra $\sC(s,t)$ are given by the assignment on objects of the multifunctor $\cat C_S\rightarrow \Sp$ and the composition and unit maps of these mapping spectra are given by the multifunctoriality.
	Moreover, a functor $\cat M_S \to \Sp$ corresponds to a spectral category $\sR$ with object set $S$ and a module $\mc M$ over $\sR$, i.e. a set of spectra $\mc M_s$ indexed on the elements of $S$, together with associative and unital action maps $\mc M_s \sma \sR_{t,s}\to\mc M_t$ for pairs of objects $s,t\in S$. Such a module is the same data as a spectrally enriched functor $\sR\to \Sp$.
	
	A map of $\sR$-modules $\sM\to \sM'$, which consists of maps of spectra $\sM_s\to \sM'_s$ for all $s\in S$ that commute strictly with the action of $\sR$, corresponds to a multinatural transformation of functors $\cat M_S \to \Sp$, which is the identity on $\cat R_S$. 
\end{exmp}

It is also possible to consider categories indexed on a set $S$ that are enriched in a multicategory $\mc M$, rather than just a symmetric monoidal category $\mc V$. For the definitions of $\mc M$-enriched category where $\mc M$ is a multicategory, and $\mc M$-enriched functors, see \cite[Def. 2.5., Def. 2.9.]{BO}. A translation of definitions gives the following lemma.

\begin{lem}\label{multi2}%\mmar{double check, but this should be true and it's the more general statement we want to use for Wald and SMC}
	If $\mc M$ is a multicategory then the category of multifunctors $\cat R_S \to \mc M$ and multinatural transformations is isomorphic to the category of $\mc M$-enriched categories with object set $S$ and maps of $\mc M$-enriched categories.
\end{lem}

%\begin{lem}\label{multi2}%\mmar{double check, but this should be true and it's the more general statement we want to use for Wald and SMC}
%	If $\mc M$ is a multicategory then each multifunctor $\cat R_S \to \mc M$ corresponds to an $\mc M$-enriched category with object set $S$. Multinatural transformations of such multifunctors give maps of $\mc M$-enriched categories.
%\end{lem}

%Moreover, a multifunctor $\cat M_S \to \mc M$ corresponds to an $\mc M$-enriched category $\sR$ with object set $S$, and an $\mc M$-enriched functor $\sR\to \mc M$. Natural transformations of such multifunctors give maps of enriched categories and  natural transformations of enriched functors.
%
%Generalizing the situation for modules enriched over monoidal categories, what we mean by an $\mc M$-enriched module over an $\mc M$-enriched category with object set $S$ is just a multifunctor $\cat M_S \to \mc M$. Natural transformations of such multifunctors give maps of enriched categories and maps of $\mc M$-enriched modules.

\begin{exmp}
Fix a finite group $G$ and let $S$ be the set of subgroups of $G$.  The $\cCMC$-enriched category $G\sE$ from \cite[Def. 7.2.]{BO} assigns each pair of subgroups $H, K \leq G$ to the cocartesian monoidal category of spans
\[ \xymatrix @R=1em @C=1em{
	& \ar[ld] S \ar[rd] & \\
	G/H && G/K
} \]
and composes these categories by taking pullbacks of the spans. This corresponds to a multifunctor $\cat R_S\to \cCMC$, assigning each pair $(H,K)$ to the above category of spans and each list of subgroups to the functor
\[ G\sE(H_0,H_1) \times \ldots \times G\sE(H_{n-1},H_n) \longrightarrow G\sE(H_0,H_n) \]
that takes the pullback of all the spans in the list. The postcomposition with the multifunctor $K^\Seg\colon \SMC \rightarrow \Sp$ to symmetric spectra constructed in \cite{EM} gives the spectral category $\GB$. % that is used in \cite{GM1} to model genuine $G$-spectra.
\end{exmp}

\begin{exmp}
	Alternatively, we could assign each pair of subgroups $H,K \leq G$ to the Waldhausen category of retractive spans $\mc S_{H,K}$, in other words spans containing $G/H \times G/K$ as a retract. %= \Rd_G(G/H \times G/K)$.   ---this might be misleading given that they have different pairings 
	The pairings outlined in \cite[Definition 4.3.]{CaryMona} extend this to a multifunctor $\cat R_S \to \Wald$. The postcomposition with the multifunctor $K^\Wald\colon \Wald \rightarrow \Sp$  to symmetric spectra constructed in \cite{inna} gives the spectral category $\GB^\Wald$.
\end{exmp}

\begin{rem}
Both examples above adopt the conventions of  \cite[\S 1.1]{GM1}, assuming that every  $G$-set $A$ is one of the sets $\{1,\ldots,n\}$ with a $G$-action given by some homomorphism $G\ra \Sigma_n$, so that the coproduct, product, and pullback are given by specific formulas that make them associative on the nose. In particular, the pullback is defined by taking a subset of the product, ordered lexicographically.  This chosen model for the pullback of $G$-sets has the slight defect that the unit span 
\[ \xymatrix @R=1em @C=1em{
& G/H \ar[ld]_-{\id} \ar[rd]^-{\id} & \\
G/H && G/H } \]
is only a left-sided unit for the multiplication on the span category, and this is rectified by whiskering the category of spans with a formal unit object $\mathbf{1}_{G/H}$ and a unique isomorphism $\mathbf{1}_{G/H} \cong (G/H)_+$. For a more detailed discussion of this, see \cite[\S 1.1]{GM1} or \cite[Remark 4.1.]{CaryMona}
\end{rem}

Our use of the parameter multicategories $\cat R_S$ and $\cat M_S$ here reinterprets some of the concepts that appear in \cite{BO} in terms of rings on many objects and modules over them as opposed to enriched categories and functors. Rather than defining $\GE$ as a category enriched in $\SMC$   and defining a module over it as a $\SMC$-enriched functor $\GE \to \SMC$, and then changing the enrichment to get a spectrally-enriched functor $\GB \to \Sp$, we define $\GE$ as a multifunctor $\cat R_S \to \SMC$, and we define a module over $G\sE$ in $\SMC$ as an extension of this to a multifunctor $\cat M_S \to \SMC$. The change of enrichment is accomplished by postcomposing this multifunctor with $K^\Seg\colon \SMC \to \Sp$ to get a multifunctor $\cat M_S \to \Sp$, which in turn corresponds to a module over $\GB$, or equivalently a spectrally enriched functor $\GB\to \Sp$. The following lemma accomplishes this translation.

%\begin{lem}\label{multi3}
%Suppose $\sR$ is a category with object set $S$ enriched over a multicategory $\mc M$\cary{propose deleting:, enriched over itself.} An $\mc M$-enriched functor $\sR\to \mc M$ corresponds to a multifunctor $\cat M_S \to \mc M$ which restricts to the multifunctor defining $\sR$ on $\cat R_S$. Natural transformations of enriched functors correspond to multinatural transformations of multifunctors.
%\end{lem}
\begin{lem}\label{multi3}
Suppose $\sR$ is a category with object set $S$ enriched over a multicategory $\mc M$, and suppose $\mc M$ is enriched over itself. The category of $\mc M$-enriched functors $\sR\to \mc M$ and enriched natural transformations is isomorphic to the category of multifunctors $\cat M_S \to \mc M$ which restrict to the multifunctor defining $\sR$ on $\cat R_S$, and multinatural transformations of multifunctors.
\end{lem}

In particular, the $\SMC$-enriched functors $\GE\to \SMC$ from \cite{BO} correspond to a multifunctors $\cat M_S\to \SMC$ which restrict to the multifunctor defining $\GE$ on $\cat R_S$. Moreover, the module over $\GB^\Wald$  from \cite{CaryMona} can be reinterpreted as a multifunctor.

\begin{exmp}
	Fix a finite group $G$ and let $S$ be the set of subgroups of $G$. The main construction  in \cite[\S 4]{CaryMona} recalled in \autoref{AGreview} below can be restated as saying that we construct a multifunctor $\cat M_S \to \Wald$ sending $(H,K)$ to the Waldhausen category $\mc S_{H,K}$ and $H$ to a category equivalent to the Waldhausen category $\Rhf_H(X)$ of $H$-equivariant homotopy finite retractive spaces over $X$ and genuine $H$-weak equivalences. Upon postcomposition with the multifunctor $K^\Wald\colon \Wald \rightarrow \Sp$ to symmetric spectra constructed in \cite{inna}, we get a multifunctor $\GB^\Wald\to \Sp$, corresponding to the module $\bA_G(X)$ over $\GB^\Wald$.
\end{exmp} %\cmar{Should we name this multifunctor?} %%% added name of module \bA_G(X); I think this is fine since we don't use the multifunctor name from here on

The multiplicative comparison of Waldhausen and Segal $K$-theory in \cite{bohmann_osorno} takes the following form. First note that the multicategory $\Wald$  is equivalent to the multicategory $\Wald_\vee$ where for each pair of objects, there is a chosen coproduct, and in all the multifunctors in the examples above we can replace $\Wald$ by $\Wald_\vee$ by making choices of coproducts in the categories $\mc S_{H,K}$ and $\Rhf_H(X)$. 
\begin{thm}[\cite{bohmann_osorno}]\label{bo_mult}
	The passage from a Waldhausen category to its underlying category with direct sum defines a multifunctor
	$$\Lambda\colon \Wald_\vee \rightarrow \cCMC_{we},$$ where $\cCMC_{we}$ denotes cocartesian monoidal categories with weak equivalences.
	There is a multinatural transformation of multifunctors 
	\[\xymatrix{
		\Wald_\vee \ar[dr]_-{K^\Wald(-)} \ar[rr]^-\Lambda && \cCMC_{we}\ar[dl]^-{K^\Seg(w-)}\\
		& \Sp &
	}\]
%	which is an equivalence for Waldhausen categories with split cofibrations. \mona{(we need to put the actual conditions here and a discussion of this)}
\end{thm}
%
%Note that $\Lambda \GB^\Wald(G/H, G/K)\simeq \GB(G/H, G/K)$ \mona{(need to prove this, BO probably will do)}, so in particular we get maps 
%$$K^\Seg( \GB(G/H, G/K))\rightarrow K^\Wald(\GB^\Wald(G/H, G/K))$$

Composing this theorem with the multifunctor $\cat M_S \to \Wald_\vee$ from above, we get a map of spectrally-enriched categories and modules over them. When the weak equivalences are  isomorphisms,  following $K$-theory tradition \cite{waldhausen}, we denote the subcategory of weak equivalences with an $i$ instead of a $w$.

\begin{cor}\label{GBtoGB}
There is a map of spectral categories $\GB \to \GB^\Wald$.
\end{cor}

\begin{proof}
Let $S$ be the set of subgroups of $G$. Recall that $\GB^\Wald$  and $\GB$ are defined by the multifunctors
$$\cat R_S \xrightarrow{\mc S} \Wald \xrightarrow{K^\Wald} \Sp\text{\ \ \  and\ \ \ } \cat R_S\xrightarrow{\GE} \cCMC \xrightarrow{K^\Seg} \Sp,$$
respectively.
%is  defined by the multifunctor $\cat R_S \xrightarrow{\mc S} \Wald \xrightarrow{K^\Wald} \Sp$, and $\GB$ is defined by the multifunctor  $\cat R_S\xrightarrow{\GE} \cCMC \xrightarrow{K^\Seg} \Sp$. 
Thus by the theorem it is enough to compare the multifunctors
$$\cat R_S \xrightarrow{\mc S} \Wald_\vee \xrightarrow{\Lambda} \cCMC_i \text{\ \ \      and\ \ \      } \cat R_S\xrightarrow{\GE} \cCMC.$$ By definition, both $i\Lambda \mc S_{H,K}$ and $G\sE(G/H, G/K)$ are the category of spans from $G/H$ to $G/K$ and isomorphisms. Composition was defined in both cases via the explicit model for the pullback given by choosing the subset of the product given by lexicographic ordering. Thus composition maps agree, and the two multifunctors agree.
\end{proof}

\begin{cor}\label{translate_modules}
Suppose $\{ M_H\}$ is a $\GB^\Wald$-module. Then, along the map of spectral categories $\GB \to \GB^\Wald$ from \autoref{GBtoGB}, there is a map of spectral Mackey functors
\[ \xymatrix{ \{K^\Seg(w\Lambda M_H)\} \ar[r] & \{K^\Wald(M_H)\}. } \]
\end{cor}
%%SPECIFIC STATEMENT WE NEED LATER ON IS:
%Along this, for any $G$-set $X$ there is an equivalence of modules $\{K^\Wald(\underline{\Rd}_H(X))\} \simeq \{K^\Seg(i\Lambda\underline{\Rd}_H(X))\}$.

\begin{proof} The statement is  immediate when we view these modules as precomposing the diagram of multifunctors in \autoref{bo_mult} with the multifunctor $\cat M_S\to \Wald_\vee$ which defines the module $\{M_H\}$. 
\end{proof}

\subsection{Equivalences of Waldhausen and Segal $K$-theory} Now we establish when these maps of spectral Mackey functors are equivalences.
%For this we will freely invoke results from \cite{waldhausen}. 
By \cite[\S 8]{waldhausen}, the map $K^\Seg(w\Lambda \sC) \to K^\Wald(\sC)$ is an equivalence if $\sC$ is a Waldhausen category with a cylinder functor, satisfying the separation and extension axiom, and which has split cofibrations up to weak equivalence. In \cite{DGM}, the authors provide a simplified argument in the case $\sC$ is an additive category. Unfortunately, these results do not apply to our case of interest. The categories of retractive finite sets and finite $G$-sets do not have a cylinder functor and are not additive. We give a modified version of Waldhausen's argument which we have learned from Badzioch and Dorabia\l{}a\footnote{Private communication}.

\begin{prop}\label{segal_equals_waldhausen}
Suppose $\sC$ is a Waldhausen category with a functorial splitting of cofibrations with respect to weak equivalences, and suppose $\sC$ satisfies the extension axiom. Then the component at $\sC$ of the natural transformation in \autoref{bo_mult}, $ K^\Seg(w\Lambda \sC) \to K^\Wald(\sC)$, is an equivalence.
\end{prop}

By a functorial splitting, we mean a choice for each cofiber sequence $a \to b \to c$ of a map $c \to b$ such that $c \to b \to c$ is the identity, such that for any weak equivalence of cofiber sequences as indicated on the left, the corresponding square on the right also commutes.
\[ \xymatrix{
	a \ar[d]^-\sim \ar[r] & b \ar[d]^-\sim \ar[r] & c \ar[d]^-\sim && b \ar[d]^-\sim \ar@{<-}[r] & c \ar[d]^-\sim  \\
	a' \ar[r] & b' \ar[r] & c' && b' \ar@{<-}[r] & c'
} \]

\begin{proof}
%\mnote{Discussion of what's in Waldhausen: Waldhausen takes the map from $wN.\sC$ to $wS.\sC$, where $wN.\sC$ is the $\Gamma$-category construction associated to the symmetric monoidal category ($\sC, \vee)$, and where you precompose with $\Delta^{op}$ to get a simplicial category and then form the nerve only with respect to weak equivalences. The geometric realization of this is Segal's construction of the first space in the spectrum indeed. He shows that sometimes this is a homotopy equivalence}
In this proof, we will freely cite results from \cite{waldhausen}. We start by recalling some notation and definitions from there that we need. For a Waldhausen category $\sC$,  the construction $w N^{Seg}_{\sbt} \sC$ (which Waldhausen denotes $wN.\sC$) is not the nerve, but the $\Gamma$-category construction associated to the symmetric monoidal category ($\sC, \vee)$,  precomposed with $\Delta^{op}$ to get a simplicial category and then forming the nerve only with respect to weak equivalences. 
The geometric realization of this is Segal's construction of the first space in the Segal $K$-theory spectrum. 
By \cite[Prop. 1.8.7.]{waldhausen}, the map $ K^\Seg(w\Lambda \sC) \to K^\Wald(\sC)$ is an equivalence if for every object $x\in \sC$, the realization $|w N^{Seg}_{\sbt}(j\colon \sC \to \sC_x)|$ is contractible, where  $\sC_x$ denotes the category of cofibrant objects under $x$ and $j$ is the functor $b \mapsto x \vee b$. 
%Here $N^{Seg}_{\sbt}$ denotes a strictification of the categorical bar construction with respect to $\vee$, not the nerve, while $\sC_X$ denotes the category of cofibrant objects under $X$ and $j$ is the functor $B \mapsto X \vee B$.
 By \cite[Prop. 1.8.9.]{waldhausen} this further simplifies to the statement that the realization $|w N^{Seg}_{\sbt}(j\colon \sC \to \sC_x)|$ is connected and that $|w N^{Seg}_{\sbt}\sC| \to |w N^{Seg}_{\sbt}\sC_x|$ is a homotopy equivalence. To show the first condition holds, we observe that since $\sC$ satisfies the extension axiom, every cofibration $x \to a$ is equivalent to $x \to x \vee a/x$, and this is connected to the initial object $x \to x$, hence the realization $|w N^{Seg}_{\sbt}(j\colon \sC \to \sC_x)|$ is connected.

For the second condition, we show the map is an equivalence on each level in the $N^{Seg}_{\sbt}$ direction. Each level is equivalent to an $n$-fold product of the map $|w_{\sbt} \sC| \to |w_{\sbt} \sC_x|$ so it suffices to show this map is an equivalence. We define a homotopy inverse by the functor $\sC_x \to \sC$ taking $x \to a$ to $a/x$. The composite $w\sC \to w\sC_x \to w\sC$ is clearly isomorphic to the identity, while the composite $w\sC_x \to w\sC \to w\sC_x$ takes $x \to a$ to $x \to x \vee a/x$. By the assumption that cofiber sequences split in a way that is functorial with respect to weak equivalences, this also admits a natural transformation in $w\sC_x$ to the identity. Therefore the map of nerves is a homotopy equivalence.
\end{proof}

\begin{rem}
The hypotheses of \autoref{segal_equals_waldhausen} do not tend to hold in algebraic examples, for instance it fails for the category of abelian groups and isomorphisms. However, it is easily checked to be true in the example of interest to us, namely the category of retractive finite sets or $G$-sets, see \autoref{cofiber_sequences_split}. Essentially, this is because an invertible matrix of the form
\[ \begin{pmatrix}
	A & B \\
	0 & C
	\end{pmatrix} \]
can have $B \neq 0$, but if it is a permutation matrix then $B = 0$. % This is the only example that we are aware of where the condition holds. \cary{No need to rub it in...}
\end{rem}

%In particular,  we obtain the following theorem.

\begin{thm}\label{GB_equivalence}
There is an equivalence of spectral categories $\GB\simeq \GB^\Wald$.
\end{thm}

\begin{proof}From the definition, cofiber sequences in $S_{H,K}$ split functorially with respect to isomorphisms, and satisfy the extension axiom. 
The theorem  follows now immediately from \autoref{GBtoGB} and \autoref{segal_equals_waldhausen}.
 
\end{proof}

By  \cite[6.1]{schwede_shipley_equivalences}, we therefore get the following corollary.

\begin{cor}\label{GBmodules}
There is a Quillen equivalence of categories of modules $$\GB-\Mod\simeq \GB^\Wald-\Mod$$.
\end{cor}
Therefore both of these categories of modules are Quillen equivalent to the category of $G$-spectra, and ``spectral Mackey functors" can mean either.

%\begin{rem}
%	The same proof also gives an equivalence of modules $$\{K^\Wald(\underline{\Rhd}_H(X))\} \simeq \{K^\Seg(i\Lambda\underline{\Rhd}_H(X))\},$$  though we will not need this here.
%\end{rem}

%\cnote{Can the def below should be formulated with values in an arbitrary multicategory instead of $\SMC$? I'm not sure.}
%\mnote{I tried thinking about it but I thought we might stick to the case we need in the ``anti" Peter spirit :). But yes, I think you could... maybe we can come back to it after and see how much more difficult it would be after we have it written for what we need and if it works the same, we can generalize. I didn't feel completely comfortable to do so now before really understanding what goes into this proof for SMC because, for example, what are the 2-cells in general, if your objects and maps are not categories and functors? We would   need some 2-category structure also in order to have 2-cells, and we would need to think carefully about what compatibilities we need. Rather than possibly writing down something wrong or nonsensical, I'd rather stick to the SMC case for now.}
%

\subsection{Pseudo linear maps of modules}
In the next section we will compare two models of equivariant $A$-theory that are modules over the ring in $\cCMC$ defined by $\GE$. We will not obtain a strict map of $\GE$-modules, but one that instead commutes with the action of $\GE$ up to coherent isomorphism. We capture this formally in the following definition.
%in the next section, we take the one that is a module over the ring in $\Wald$ defined by $\mc S_{(-),(-)}$ and push it forward to the  

Let $S$ be a set, and suppose $R\colon \cat R_S\to \cCMC$ is a multifunctor encoding a ``ring" $R$ of cocartesian monoidal categories with object set $S$, and $M, M'\colon \cat M_S\to \cCMC$ be multifunctors that restricts to $R$ on $S\times S$, encoding  ``modules" over $R$. We use $m$ to denote the multiplication maps for the ring and $a$ to denote the action maps of the ring on the module.

\begin{defn}\label{pseudolinear}
%	(Pseudo linear map of modules.)
	A \emph{pseudo linear} map of modules $f\colon M\to M'$ is a collection of coproduct-preserving functors $f\colon M_s\to M'_s$, one for each object $s\in S$, and for every pair $s,t \in S$ an invertible natural transformation $\theta_{s,t}$ making the following diagram commute.
	\[\xymatrix{
		R_{s,t}\times M_t \ar[d]_-{\id\times f_t} \ar[r]^-a  \drtwocell<\omit>{<0> \ \ \   \theta_{s,t}}  & M_s \ar[d]^-{f_s} \\%
		R_{s,t}\times M'_t  \arrow[r]^-a & M'_s
	}\]
	Note that this makes $\theta_{s,t}$ is a natural transformation of symmetric monoidal functors, in each slot separately.
%	\cmar{In all our examples, the symmetric monoidal structure is by coproduct, and the monoidal structure on each functor is by the canonical map, and in that setting, this condition on the natural transformation becomes automatic.}
	% FAILED TIKZ ATTEMPTS
	%\mona{tikz:}
	%\[ \begin{tikzcd}
	%R_{s,t}\times M_t \arrow{r}{} \arrow[swap]{d}{} & M_s \arrow{d}{} \\%
	%R'_{s,t}\times M'_t \arrow[ur,shorten <>=15pt,Rightarrow, "\theta"] \arrow{r}{}& M'_s
	%\end{tikzcd}
	%\]
	%
	%\mona{tikz:}
	%\begin{center}
	%\begin{tikzcd}[row sep=2.5em]
	%R_{s,t}\times R_{t,r}\times M_r \arrow[rr, "m\times \id"] \arrow[dr,swap,"\id \times a"] \arrow[dd,swap,""] &&
	%  R_{s,r}\times M_r \arrow[dd,swap,"\id \times f" near start] \arrow[dr,"a"] \\
	%& R_{s,t}\times M_t \arrow[rr, swap, crossing over,"a" near start] &&
	% M_s  \arrow[dd,""] \\
	%R'_{s,t}\times R'_{t,r}\times M'_r  \arrow[rr,"m\times \id" near end] \arrow[dr,swap,"\id \times a"] \arrow[ur,shorten <>=20pt,Rightarrow, "\theta"]  &&  R'_{s,r}\times M'_r \arrow[dr,swap,"a"] \\
	%& R'_{s,t}\times M'_{t} \arrow[rr,"a"] \arrow[uu,<-,crossing over,"" near end]&& M_s'
	%\end{tikzcd}
	We also require that the isomorphisms $\theta_{-,-}$ respect units and associativity  in the sense that the following diagrams commute. All the unlabeled 2-cells are equalities.\\
	
	(Unit condition: $\theta_{s,s} = \id$)
	\[ \xymatrix @R=1em {
		M_s\ar[rd] \ar[rr]^= \ar[dd]_-{f_s} & &  M_s \ar[dd]^-{f_s} \\
		&  R_{s,s}\times M_s\ar[dd]^-(0.3){f_s}  \xtwocell[dr]{}<\omit>{<0> \, \theta} \ar[ru]_-a &&\\
		M_s'\ar[dr] \ar[rr]^(0.4){=}|\hole &&M'_s\\
		& R_{s,s}\times M_s'   \ar[ru]_-{a} &&
	}\]
	
	(Associativity condition: $\theta_{t,r} \circ \theta_{s,t} = \theta_{s,r}$)
	\[\xymatrix @R=1em @!C=6em {
		R_{s,t}\times R_{t,r}\times M_r \ar[rr]^-{m\times \id} \ar[dr]_-{\id \times a} \ar[dd]_-{1 \times 1 \times f_s} &&
		R_{s,r}\times M_r \ar[dd]|\hole^-(0.3){1 \times f_s} \ar[dr]^-a \xtwocell[dddr]{}<>{\ \ \ \ \theta_{s,r}} \\
		\drtwocell<\omit>{<0> \, \ \ \theta_{t,r}}  & R_{s,t}\times M_t \ar[rr]^(.4)a \ar[dd]^-(0.7){1 \times f_s}   &  
		&   M_s     \ar[dd]^-{f_s} \\
		R_{s,t}\times R_{t,r}\times M'_r  \ar[rr]^(0.65){m\times \id}|\hole \ar[dr]_-{\id \times a}  & \drrtwocell<\omit>{<-1> \,\  \ \theta_{s,t}} &  R_{s,r}\times M'_r \ar[dr]^{a} \\
		& R_{s,t}\times M'_{t} \ar[rr]^-a  & &  M_s'
	}\]
\end{defn}

Next we prove that a pseudo-linear map of modules can be modified up to equivalence to a strict map.

\begin{con}\label{underline}
	Let $R$ be a ring in $\cCMC$ with object set $S$, and let $M$ be a module over $R$. We define a new $R$-module $\underline{M}$ as follows. For $s\in S$, the objects of $\underline{M}_s$ are given by 
	$$\mathrm{ob}\underline{M}_s=\coprod_t \mathrm{ob}R_{s,t}\times \mathrm{ob}M_t, $$ and a morphism $(r,m)\to (r',n)$ for $r \in R_{s,t}$, $r'\in R_{s,p}$ and $m \in M_t$, $n\in M_p$ is given by a morphism $rm\to r'n$ in $M_s$.
\end{con}

\begin{prop}\label{tildefattening}
	Suppose $M$ is a an $R$-module in $\cCMC$. Then $\underline{M}$ is also an $R$-module, and there is an equivalence of $R$-modules $\underline{M} \overset\sim\to M$.
\end{prop}

\begin{proof}
	Each $\underline{M}_s$ is a cocartesian monoidal category because it is equivalent to $M_s$. We define the action maps $R_{s,t}\times  \underline{M}_t\to\underline{M}_s$ on objects by
	$$(r,\ (s, m))\mapsto (rs, m),$$ where $r\in R_{s,t}$, $s\in R_{t, r}$, $rs\in R_{s,r}$ and $m\in M_r$. This assignment extends on morphisms in the obvious way. It is strictly associative, unital, and coproduct preserving in each slot since the action maps $R_{s,t}\times M_s\to M_t$ are.
%\end{proof}
%
%\begin{prop}
%	Let $R$ be a ring in $\cCMC$ with object set $S$. Let $M$ be a module over $R$ in $\cCMC$.  Then $M$ and $\underline{M}$ are equivalent as $R$-modules.
%\end{prop}
%
%\begin{proof}

	Note that the functors $\underline{M}_s\to M_s$ by $(s,m)\mapsto sm$ are equivalences of categories by definition, and therefore are coproduct preserving as well. We have a strictly commutative diagram
		\[\xymatrix{
		R_{s,t}\times \underline{M}_t \ar[d] \ar[r]   & \underline{M}_s \ar[d] \\%
		R_{s,t}\times M_t  \arrow[r]& M_s,
	}\] so these equivalences of categories specify  a strict map of $R$-modules.
\end{proof}

\begin{thm}\label{pseudo_can_be_strictified}
%	Let $R$ be a ring on object set $S$ in $\cCMC$.
	A pseudo linear map of $R$-modules can be modified up to equivalence to a strict map of $R$-modules.
\end{thm}

\begin{proof}
	We will show that given a pseudo linear map of $R$-modules $f\colon M\to M'$, there is a strict map of $R$-modules $\underline{f}\colon \underline{M}\to \underline{M'}$. On objects, $\underline{f}(r,m)=(r, f(m))$. On morphisms, for a map  $rm\to sn$, we define $\underline{f}$ as the composite
	$$\xymatrix{ rf(m)\ar[r]_-{\cong}^-\theta &f(rm)\ar[r] & f(sn)\ar[r]_-{\cong}^-{\theta^{-1}} &sf(n).}$$
	We claim that for every $s,t\in S$ there are commutative diagrams
			\[\xymatrix{
		R_{s,t}\times \underline{M}_t \ar[d] \ar[r]   & \underline{M}_s \ar[d] \\%
		R_{s,t}\times \underline{M'}_t  \arrow[r]& \underline{M'}_s.
	}\] 
	On objects, it is clear: starting with $r, (s,m)\in R_{s,t}\times \underline{M}_t$, both composites give $rs, f(m)$. On morphisms, given $r_1\to r_2$ in $R_{s,t}$ and $(s_1,m_1)\to (s_2,m_2)$ in $\underline{M}_t$, namely a map $s_1m_1\to s_2m_2$ in $M_t$, the following diagram, where the top represents the top down composite and the bottom represents the down right composite in the diagram above, commutes by the hypotheses for $\theta$ required in \autoref{pseudolinear}. 
	
	 $$\xymatrix{ r_1s_1f(m_1)\ar[r]_-{\cong}^-\theta &f(r_1s_1m_1)\ar[r] & f(r_2s_2m_2)\ar[d]^-\cong_{\theta^{-1}}\ar[r]_-{\cong}^-{\theta^{-1}} &r_2s_2f(m_2)\\
	r_1 s_1f(m_1)\ar[u]_-=\ar[r]_-{\cong}^-\theta&r_1f(s_1m_1)\ar[u]^-\theta_\cong \ar[r] & r_2f(s_2m_2)\ar[r]_-{\cong}^-{\theta^{-1}} &r_2s_2f(m_2)\ar[u]_-=}$$
	Finally, $\underline{f}$ is coproduct-preserving because on the equivalent subcategory on objects of the form $(1,m)$, it is simply $f$, which is assumed to be coproduct-preserving.
\end{proof}

\begin{rem}\label{thickening_preserves_functoriality}
Suppose $f\colon M\to M'$ is actually a strict map of $R$-modules, namely that all $\theta=\id$. Then the definition of $\underline{f}$ restricts to the functor induced by $f$, which on morphisms $rm\to sn$ is just given by $f(rm)\to f(sn)$ since $f$ commutes with the $R$-action. It is also easy to check that this procedure is functorial (preserves compositions and units) if $f$. Therefore if we have a functor from some category to $R$-modules and strict maps, postcomposing with this thickening gives an equivalent functor into $R$-modules and strict maps.
\end{rem}

\begin{rem}\label{different_strictifications}
	In \cite [Proposition 4.13]{CaryMona} we used the $\underline{M}$ construction  to strictify a ``pseudo module" $M$, namely an action of $R$ on $M$ that is only associative and unital up to coherent isomorphism.  	We expect that with more effort one could integrate these strictifications into a single definition, where the multiplication on $R$, the action on $M$, and the map $M \to M'$ all satisfy associativity up to coherent isomorphisms that are suitably compatible, and that a single construction can be shown to  strictify the multiplication, the action maps and the map $M \to M'$ simultaneously.

%	The strictification  construction of $\underline{M}$ from a certain module $M$ in $\Wald$ was also used in the proof of \cite[Proposition 4.13]{CaryMona}  for a different purpose. In \cite{CaryMona} we were strictifying a ``pseudo module", an action of $R$ on $M$ that was only associative and unital up to coherent isomorphism. We expect that with more effort one could integrate these into a single definition, where the multiplication on $R$, the action on $M$, and the map $M \to M'$ all satisfy associativity up to coherent isomorphism, and a single construction strictifies the multiplication, the action maps and the map $M \to M'$ simultaneously.
%	More generally, in CM, we have proved:\\
%		suppose $R$ is a ring on many objects whose composition pairings are strictly associative and unital. If we have an action which is associative only up to coherent isomorphism, i.e., we have compatible associativity isomorphisms
%		$$\xymatrix{
%			R_{s,t}\times R_{t,r}\times M_r \ar[d]_-{\id\times a} \ar[r]^-{m\times \id}  \drtwocell<\omit>{<0> \, \tau}  & R_{r,s}\times M_r     \ar[d]_-a\\
%			R_{s,t}\times M_t \ar[r]^-a &M_s
%		}$$
%		then we can define a strict action maps $$R_{s,t}\times \times \widetilde{M}_t \to \widetilde{M}_S$$ which are strictly associative. By definition of $\widetilde{M}$ and the $R$-action on it, the action is strictly associative on objects, because it uses the strictly associative multiplication pairings on $R$.
\end{rem}

Lastly, we check naturality properties of this strictification procedure. 

\begin{df}\label{pseudo_is_natural}
Suppose $f\colon M\to M'$ and $g\colon N\to N'$ are pseudo-linear transformations of $R$-modules, where $R$ is a ring on object set $S$ in $\cCMC$. We say that linear maps of modules $\alpha\colon M\to N$ and $\beta\colon M'\to N'$ specify a  \emph{natural map of pseudo-linear transformations} if the following cube of $2$-cells commutes (the top, bottom, and sides of the cube are equalities, only the front and back have $\theta_{s,t}$).
	\[\xymatrix @R=1em @!C=6em{
		R_{s,t}\times M_t \ar[rr]^-{a} \ar[dr]_-{1\times\alpha} \ar[dd]_-{1 \times  f_t}
		\xtwocell[ddrr]{}<>{<-3>\ \ \ \ \theta_{s,t}} & &
		 M_s \ar[dd]|\hole^-(0.3){ f_s} \ar[dr]^-\alpha \\
		 & R_{s,t}\times N_t \ar[rr]^(.4)a  \xtwocell[ddrr]{}<>{<3>\ \ \ \ \theta_{s,t}}  \ar[dd]^-(0.7){1 \times g_t}  
		  & &   N_s     \ar[dd]^-{g_s} \\
		R_{s,t}\times M'_t  \ar[rr]^(0.65){a}|\hole \ar[dr]_-{1\times \beta}  & &   M'_s \ar[dr]^{\beta} \\
		& R_{s,t}\times N'_{t} \ar[rr]^-a  & &  N_s'
	}\]
	In particular, we have $g_s \alpha=\beta f_s$ for all $s\in S$.
	
\end{df}

\begin{lem}\label{natural_thickens_to_natural}
Suppose that  $f\colon M\to M'$ and $g\colon N\to N'$ are pseudo-linear transformations of $R$-modules and that $\alpha\colon M\to N$ and $\beta\colon M'\to N'$ specify a natural map of pseudo-linear transformations. Then the following diagram of linear maps of modules strictly commutes.
$$\xymatrix{\underline{M} \ar[r]^-{\underline{\alpha}} \ar[d]_-{\underline{f}} & \underline{N} \ar[d]^-{\underline{g}}\\
\underline{M'}\ar[r]^-{\underline{\beta}} &\underline{N'}
}$$
\end{lem}

\begin{proof}
Let $s\in S$. We want to show that the square 
$$\xymatrix{\underline{M_s} \ar[r]^-{\underline{\alpha}} \ar[d]_-{\underline{f}} & \underline{N_s} \ar[d]^-{\underline{g}}\\
\underline{M_s'}\ar[r]^-{\underline{\beta}} &\underline{N_s'}
}$$
commutes. On objects, it commutes by definition, since we are assuming  $g\alpha=\beta f$. For a morphism $\gamma\colon rm\to sn$, we need to show that the following rectangle commutes
$$\xymatrix{rg\alpha(m)\ar[r]^-\theta_-\cong \ar[d]^-= & gr\alpha(m)\ar[r]^-{g\alpha(\gamma)} \ar[d]^-=& gs\alpha(n)  \ar[r]^\theta_-\cong \ar[d]^-= & sg\alpha(n) \ar[d]^=  \\
\beta(rf(m))\ar[r]^-{\beta(\theta)}_-\cong & \beta(f(rm))\ar[r]^-{\beta f(\gamma)} & \beta (f(sn))  \ar[r]^\theta_-\cong & \beta (sf(n))  }$$

The middle square again commutes because we are assuming $g\alpha=\beta f$, and the commutativity of left and right squares follows from the assumed commutativity of the cube in \autoref{pseudo_is_natural}. 
\end{proof}

%\section{Description of $\Sigma_G^\infty X_+$ as a spectral Mackey functor}
\section{The map $\Sigma_G^\infty X_+\to \bA_G(X)$}\label{map}

In this section we construct a map of $G$-spectra $\Sigma_G^\infty X_+\rightarrow \bA_G(X)$ that respects the tom Dieck splittings on the fixed points. Since $\bA_G(X)$ arises from a spectral Mackey functor, we do this by constructing a map of spectral Mackey functors.

The main difficulty is that $\bA_G(X)$ arises from a module over $\GB^\Wald$, whereas the suspension spectrum $\Sigma_G^\infty X_+$ is defined in \cite{GM3} as a certain module over $\GB$. We therefore have to define a new module over $\GB^\Wald$ that models $\Sigma_G^\infty X_+$, and map it into $\bA_G(X)$. Defining the module is not so hard, but proving that it models $\Sigma_G^\infty X_+$ takes a fair amount of work, and is the main theorem of this section (\autoref{it_models_suspension_spectrum}). The proof builds on the result from \cite{BD} that identifies the spectra for each value $H$ of the module,
%of the suspension $G$spectrum of $X$ in such a way that its map to the fixed points of $\bA_G(X)$ is compatible with the tom Dieck splittings,
on \autoref{bo_mult} from \cite{bohmann_osorno} and on  the coherence machinery of \autoref{pseudo_can_be_strictified} to relate it back  to the $\GB$-module from \cite{GM3} that models $\Sigma_G^\infty X_+$.

\subsection{Review of the definition of $\bA_G(X)$}\label{AGreview}

We first recall our construction of equivariant $A$-theory $\bA_G(X)$ for a $G$-space $X$. We refer the reader to \cite{CaryMona} for the details of the construction.
\begin{df}\label{retractive_h_spaces}(\cite[Definition 4.13]{CaryMona})
%\mmar{here we now have a clash with the last section, where we claim to be using the definition of $A$-theory as $K(R_f(X))$ in order to avoid proving the commutation with colimits for the $R_{hf}$ version.}
For each subgroup $H \leq G$, let $R_H(X)$ be the category of %relative $H$-cell complexes $X \ra Y$. 
$H$-equivariant retractive spaces $Y$ over $X$. % with $H$-equivariant inclusion $i_Y$ and retraction $p_Y$. The morphisms are $H$-equivariant maps between these.
The weak equivalences are those inducing weak equivalences on the fixed points $Y^L$ for all $L \leq H$. The cofibrations are the maps $Y \ra Z$ with the $H$-equivariant FHEP: there is an $H$-equivariant, fiberwise retract
\[ Z \times I \ra Y \times I \cup_{Y \times 1} Z \times 1. \]
In particular, when $L \leq H$, the $L$-fixed points of a cofibration are a cofibration in $R(X^L)$. Let $\Rhf_H(X)$ be the subcategory of objects which are homotopy equivalent to finite relative $H$-cell complexes $X\ra Y$.\footnote{We have switched from the upperscript notation $R^H(X)$  from \cite{CaryMona} to the subscript notation $R_H(X)$  in order to match with the notation in \cite{BD}. We have also changed the definition of $\Rhf(X)$ slightly from \cite{CaryMona}, where it had all retracts in the homotopy category of $R(X)$ of finite relative cell complexes $X \ra Y$. It would have been better to call that notion ``homotopy finitely dominated spaces" and denote it $R^{hfd}$. It is easy to check that the results from \cite{CaryMona} apply to homotopy finite spaces in this sense, we only have to skip the step in the argument where we pass to retracts. The inclusion $\Rhf(X) \subseteq R^{hfd}(X)$ also gives an isomorphism on higher $K$-groups, by a cofinality argument, and therefore the statement of \autoref{parametrized_h_cobordism_theorem} is unaffected by which definition we use.} 
%retracts in the homotopy category of $R_H(X)$ of finite relative $H$-cell complexes $X \ra Y$.\footnote{We have switched from the upperscript notation $R^H(X)$  from \cite{CaryMona} to the subscript notation $R_H(X)$  in order to match with the notation in \cite{BD}.}
\end{df}

The span categories $\mc S_{H,K}$ act on these categories of retractive spaces, in the sense that there are bi-exact functors $R_H(X) \times \mc S_{H,K} \to R_K(X)$, preserving the subcategory of homotopy finite retractive spaces. 

To define the action of a span, we change the categories $R_H(X)$ up to equivalence. For any left $G$-space $X$ we give the category $R(X)$ of non-equivariant retractive spaces a left $G$-action by ``conjugation,'' sending $X \overset{i}\to Y \overset{p}\to X$ to $X \overset{i \circ g^{-1}}\to Y \overset{g \circ p}\to X$. We let $\tG$ be the contractible category with object set $G$, with $G$ acting on the objects by left multiplication. For each $H \leq G$ we construct the homotopy fixed point category
\[ R(X)^{hH} := \Cat(\tG\times G/H, R(X))^G. \]
An object in this category consists of a $G$-equivariant functor $F\colon \tG \times G/H \to R(X)$. For each $(g_1,g_2H) \in \tG \times G/H$, let $F(g_1,g_2H)$ denote the total space $Y$ of the retractive space in the image of this functor. Since the $G$-action on $R(X)$ preserves this object, the $G$-equivariance of $F$ implies that
\[ F(g_1,g_2H) = F(gg_1,gg_2H) \]
for all $g \in G$. (Only the map to $X$ changes.) In particular there is an equality $F(e,H) = F(h,H)$ for all $h \in H$. We can therefore define an action of $h$ on $F(e,H)$ by composing this equality with the action of $F$ on one of the morphisms in $\tG$:
\[ \xymatrix @C=4em{ F(e,H) \ar@{=}[r] & F(h,H) \ar[r]^-{F(h \to e,H)} & F(e,H). } \]
This gives an $H$-action on $F(e,H)$, making it into a retractive $H$-space over $X$.
\begin{lem}\label{functor_to_space}(\cite[Prop 3.1]{CaryMona})
	This gives an equivalence of categories $R(X)^{hH} \simeq R_H(X)$.
\end{lem}

To each map of finite $G$-sets $p\colon S \to T$ we consider the pullback functor
\[ p^*\colon \Cat(\tG\times T, R(X))^G \longrightarrow \Cat(\tG\times S, R(X))^G \]
that simply composes with $p$, and its left adjoint pushforward or transfer map
\[ p_!\colon \Cat(\tG\times S, R(X))^G \longrightarrow \Cat(\tG\times T, R(X))^G. \]
We define the action of the span $G/H \xleftarrow{p} S \xrightarrow{q} G/K $ on the functor $F \in R(X)^{hH}$ by
\[ F * S := q_!p^*F \in R(X)^{hK}. \]
Explicit formulas for these maps appear in \cite[\S 4]{CaryMona}. In the special case of a span of the form $G/H \xleftarrow{p} G/L\xrightarrow{q} G/K $, this action can be concretely described by the following geometric formulas, see \cite[Prop. 4.14]{CaryMona}:
\begin{itemize}
	\item If $L = K$ and $L \leq H$, the span acts by the functor $p^*\colon R_H(X) \rightarrow R_L(X)$ that restricts the $H$-action to an $L$-action.
	\item If $L = K$ and $L$ is conjugate to $H$, the span acts by a functor $c^*\colon R_H(X) \rightarrow R_L(X)$ that transforms the $H$-action to an $L$-action by conjugating each element of $L$.
	\item If $L = H$ and $L \leq K$, the span acts by the functor $q_!\colon R_L(X) \rightarrow R_K(X)$ that is the left adjoint to restriction. It applies $K \times_L (-)$ to the retractive space $Y$, but collapses the subspace $K \times_L X$ back down to $X$.
\end{itemize} 

%Since any $G$-map between orbits is a composite of a subgroup inclusion and a conjugation isomorphism, and since the span action commutes with coproducts, these three formulas combine to give the description of any span.

Returning to general spans, we use Beck-Chevalley isomorphisms to show that their action is associative and unital up to canonical isomorphism. We then replace $R(X)^{hH}$ with the  thickened version $\underline{R(X)^{hH}}$ as defined in \autoref{underline}. Explicitly, an object is a triple $(J,F,S)$ of a subgroup $J \leq G$, a functor $F \in R(X)^{hJ}$, and a span $S \in \mc S_{J,H}$, morphisms defined as if this object were the object $F * S \in R(X)^{hH}$, cf \autoref{different_strictifications}. The action of spans is defined so that
\begin{equation}\label{2_cell_to_strictified}
\xymatrix @R=2em @C=3em{
	\mc S_{H,K} \times \underline{R(X)^{hH}} \ar[d]_-\sim \ar@{-->}[r] & \underline{R(X)^{hK}} \ar[d]^-\sim \\
	\mc S_{H,K} \times R(X)^{hH} \ar[r] & R(X)^{hK}
}
\end{equation}
commutes up to canonical (Beck-Chevalley) isomorphism. Then the action of spans is strictly associative and unital.

Finally, we restrict to the subcategory $\underline{R(X)^{hH}_{hf}}$ consisting of those objects whose image in $R_H(X)$ is homotopy finitely dominated. Using the above formulas from \cite[Prop. 4.14]{CaryMona}, we directly check that the action of each span preserves this subcategory, so the subcategories also form a module of Waldhausen categories over $\mc S$. Taking Waldhausen $K$-theory gives a module over $\GB^\Wald$, and the associated $G$-spectrum we call $\bA_G(X)$.

\subsection{Homotopy discrete retractive spaces}
Next we introduce the categories of discrete and homotopy discrete retractive spaces.
%two categories $\Rd$ and $\Rhd$, used respectively in \cite{GM3} and in \cite{BD} to construct suspension spectra. %  These sources both refer to their categories as $\sF$, but their categories are very different. Since we need to use both of them, we adopt the notation $\Rd$ and $\Rhd$.

\begin{defn}\label{discrete}
(Discrete retractive spaces.)
For a $G$-space $X$, let $\Rd_G(X)$ be the subcategory  of $R_G(X)$ of $G$-spaces  $X\xrightarrow{i} Y \xrightarrow{p} X$ of the form $Y= i(X) \sqcup S$ for a finite $G$-set $S$. Morphisms are $G$-maps over and under $X$. A map $Y\rightarrow Y'$ is a weak equivalence if it is an isomorphism and a cofibration if it is injective.\footnote{We note that the category $\Rd_G(X)$, viewed as a permutative category, is called $\sF_G(X)$ in \cite{GM3}.}%\mona{--viewed as a permutative category, this is  GM's  $\sF_G(X)$.} %\cmar{I replaced $X$ with $i(X)$ so that this subcategory is closed under the $G$-action.}
\end{defn}

\begin{defn}\label{homotopydiscrete}
(Homotopy discrete retractive spaces.)
For a $G$-space $X$, let $\Rhd_G(X)$ be the subcategory of $R_G(X)$ of retractive $G$-spaces $X\xrightarrow{i} Y \xrightarrow{p} X$ of the form $Y=Y_X \sqcup  \bigsqcup_{i=1}^n (D_i \times G/H) $, where the $D_i$ are contractible spaces with trivial action and $G$ acts diagonally on $D_i\times G/H$, such that
\begin{itemize}
	\item $i(X)\subseteq Y_X$,
	\item the inclusion $i(X) \to Y_X$ is a $G$-homotopy equivalence,
	\item and the map $i\colon X\rightarrow Y$ has the $G$-homotopy extension property (HEP).  %\mona{--this is  BD's $\sF_G(X)$}
\end{itemize}
Morphisms are $G$-maps over and under $X$. A map $Y\rightarrow Y'$ is a weak equivalence if it is a weak $G$-homotopy equivalence, as in \autoref{retractive_h_spaces}.
%\cmar{There are several notions that are equivalent - strong equivalence on total space, weak equivalence after forgetting $G$-action, or just isomorphism on $\pi_0$.}
A map is a cofibration if it is has the $G$-HEP and is injective on $\pi_0$.\footnote{The category $\Rhd_G(X)$ is called $\sF_G(X)$ in \cite{BD}. This and the previous footnote show why we stayed away from the name $\sF_G(X)$.}
\end{defn}

When $G = 1$ we drop it from the notation. The categories $\Rd_G(X)$ and $\Rhd_G(X)$ are Waldhausen and the inclusions $\Rd_G(X) \to \Rhd_G(X) \to \Rhf_G(X) \to R_G(X)$ are exact.

\begin{rem}
	In the definition of $\Rhd$, we cannot insist that $Y_X$ is always isomorphic to $X$, i.e. the retractive spaces are all of the form $i(X) \sqcup  \bigsqcup_{i=1}^n D_i $, because then the pushout axiom for a Waldhausen category would fail.
%	when replacing finite sets by contractible spaces in the preceding definition, it is also necessary to ask for the retractive spaces to not only be of the form $Y_X \sqcup  \bigsqcup_{i=1}^n D_i $ where $Y_X$ is homotopy equivalent to $X$ as opposed to just 
\end{rem}

Now for each $H \leq G$, pull back these full subcategories and their Waldhausen structures along the equivalences of categories   % \mmar{I commented out the part with $R_f$ since we don't need it anymore with the new meaning for $\Rhf$.}
\[ \xymatrix{ \underline{R(X)^{hH}} \ar[r]^-\sim & R(X)^{hH} \ar[r]^-\sim & R_H(X) } \]
to get new Waldhausen categories
\[ R(X)^{hH}_\delta \subseteq R(X)^{hH}_{h\delta} %\cary{\subseteq R(X)^{hH}_{f}} 
\subseteq R(X)^{hH}_{hf} \subseteq R(X)^{hH}, \]
\begin{equation}\label{strictified_inclusions}
\underline{R(X)^{hH}_\delta} \subseteq \underline{R(X)^{hH}_{h\delta}} %\cary{\subseteq \underline{R(X)^{hH}_{f}}} 
\subseteq \underline{R(X)^{hH}_{hf}} \subseteq \underline{R(X)^{hH}}.
\end{equation}
In the discrete case, this subcategory can also be described as the homotopy fixed points of $\Rd(X)$:
\begin{lem}\label{Rdhtpyfixedpts}
	For each $G$-space $X$, there is an equivalence of Waldhausen categories $\Rd(X)^{hH}\simeq \Rd_H(X)$.
\end{lem}
In general, however, performing these operations in different orders gives different categories. For instance, $(\Rhd(X))^{hH}$ differs from $R(X)^{hH}_{h\delta}$ because their objects have different finiteness conditions and different Waldhausen structures (weak equivalences and cofibrations). This is the main reason why the following proof and the proof of \cite[Prop. 4.17]{CaryMona} are done ``by hand.''
%When $X$ is a $G$-space, the $G$-action on $R(X)$ defined in \cite[\S 3.1.]{CaryMona} restricts to a $G$-action on $\Rd(X)$ and $\Rhd(X)$. Thus we can consider the homotopy fixed points $\Rd(X)^{hH}\simeq \Cat(\tG, \Rd(X))^H$.
%
%
%The analogue of \autoref{Rdhtpyfixedpts} does not hold for $\Rhd$. It is different than the distinction between $R(X)^{hG}$ and $R_G(X)$ from \cite{CaryMona}, which are equivalent as categories and the difference lies in the Waldhausen structure. Even if we disregard the Waldhausen structure, the homotopy fixed point category $\Rhd(X)^{hG}$ has as objects retractive $G$-spaces of the form $Y=Y_X \sqcup  \bigsqcup_{i=1}^n (D_i \times G/H) $, but where the inclusion $X\rightarrow Y$ is a $G$-map which is only a cofibration and not a $G$-cofibration, and $r\colon Y_X\rightarrow X$ is a $G$-map which is only a nonequivariant homotopy inverse to $i$. 

%\cnote{The point of this remark is that we can't abstractly deduce that the $\Rhd_H(X)$ fit together into a categorical Mackey functor, we have to check by hand that they inherit that structure from $R_H(X)$.}
%\mona{I didn't include this remark to avoid possible confusion: we want the Waldhausen structure from $R_H$ not $R^{hH}$ in order to get the submodule of $A_G$ anyway}

%Recall from \cite{CaryMona} that the spectra $\{K(R_H(X)\}$ do not quite form a $\GB^\Wald$ module, but a categorical strictification of the action of spans on these categories, produces a collection of spectra $\{K(\underline{R}_H(X)\}$ that does.
\begin{prop}\label{inclusions}%\mmar{discuss this, I don't think saying ``Strictifying and applying $K^\Wald(-)$ to" is accurate in light of the above remark about lemma 3.5}
%Strictifying and applying $K^\Wald(-)$ to 
%\cmar{Fixed now I think}\mmar{this is good}
The exact inclusions of Waldhausen categories in \autoref{strictified_inclusions} give maps of $\mc S$-modules in Waldhausen categories, therefore after applying $K^\Wald$ they give maps of $\GB^\Wald$-modules.
\end{prop}

\begin{proof}
It is enough to show that for a span $G/H \xleftarrow{p} S\xrightarrow{q} G/K $, the action
\[ (-*S)\colon R_H(X) \rightarrow R_K(X), \]
when restricted to $\Rhd_H(X)$ lands in $\Rhd_K(X)$, and similarly for $\Rd_H(X)$ and $\Rd_K(X)$. We will only discuss the case of $\Rhd$ in detail, because the argument for $\Rd$ is similar and much simpler.
As in \cite[Prop. 4.14]{CaryMona}, it is enough to consider the case where $S$ is an orbit $G/L$, and then this further breaks down into just three cases. %and since then $p$ and $q$ are composites of subgroup inclusions and isomorphisms, it is again enough to consider the three cases analyzed in that proof, and  to show that in each case, the formula we identified gives the desired restriction to the subcategories of homotopy discrete retractive spaces.

{\bf Case 1.} $L = K$ and $L \leq H$. The span acts by the functor $p^*\colon R_H(X) \rightarrow R_L(X)$ that restricts the $H$-action to the action of $L$. When we restrict the $H$-action on an object of the form  $Y=Y_X \sqcup  \bigsqcup (D_i \times H/H_\alpha)$, since the orbits $H/H_\alpha$ are isomorphic as $L$-sets to disjoint unions of orbits  $L/L_\alpha$, we clearly get an $L$-space of the form $Y= Y_X \sqcup  \bigsqcup (D_i \times L/L_\alpha)$, where  $i\colon X \leftrightarrows Y_X \colon\! \!  r$ are inverse $L$-homotopy equivalences,  and an $H$-cofibration $X\rightarrow Y$ is an $L$-cofibration for any $L\leq H$. Therefore $p^*$ restricts to a functor $\Rhd_H(X) \to \Rhd_L(X)$.

{\bf Case 2.} $L = K$ and $L' = H$ are conjugate. Fix a $g$ such that $L' = gLg^{-1}$ and let $c\colon G/L \congar G/L'$ be the isomorphism of $G$-sets given by $hL \mapsto hg^{-1}L'$. Then the span acts by the functor $c^*\colon R_{L'}(X) \rightarrow R_L(X)$ sending the $L'$-equivariant retractive space $(Y,i,p)$ to the retractive space $(Y,i \circ g,g^{-1} \circ p)$, with each element $\ell \in L$ acting on $Y$ though the group isomorphism $L \to L'$, $\ell \mapsto g\ell g^{-1} \in L'$. Again, when we change the $L'$-action on a  space $Y$ of the form $Y=Y_X \sqcup  \bigsqcup (D_i \times L'/L'_\alpha)$ in $\Rhd_{L'}(X)$ to this $L$-action, we get an $L$-space of the form $Y=Y_X \sqcup  \bigsqcup (D_i \times L/L_\alpha)$. We deduce that the inclusion $(i\circ g^{-1})(X)\subseteq Y_X$ is an $L$-equivariant homotopy equivalence with the $L$-equivariant homotopy extension property from the fact that the same is true for $i(X) \subseteq Y_X$ with $L'$. Therefore $c^*$ restricts to a functor $c^*\colon \Rhd_{L'}(X) \rightarrow \Rhd_L(X)$.

{\bf Case 3.} $L \leq K$ and $L = H$. The span acts by the functor $q_!\colon R_L(X)\rightarrow R_K(X)$ that on each retractive $L$-equivariant space $Y$ is the pushout \[ \xymatrix{
K \times_L Y \ar[r] & q_!Y \\
K \times_L X \ar[r] \ar[u] & X. \ar[u] } \] 

\noindent If $Y=Y_X \sqcup  \bigsqcup (D_i \times L/L_\alpha)$ then
\begin{eqnarray*}
	K\times_L Y &\cong& \left( K\times_L Y_X \right)  \sqcup  \bigsqcup (D_i\times K/L_\alpha), \\
	q_!Y &\cong& \big(\left( K\times_L Y_X \right) \cup_{(K\times_L X)} X\big)\sqcup  \bigsqcup (D_i\times K/L_\alpha) \\
	&\cong& q_!Y_X\sqcup  \bigsqcup (D_i\times K/L_\alpha) \\
\end{eqnarray*}
Since $K\times_L -$ preserves equivariant cofibrations and homotopy equivalences, so does $q_!$, and therefore $X \to q_!Y_X$ is an equivariant cofibration and homotopy equivalence. Therefore $q_!$ restricts to a functor $q_!\colon \Rhd_L(X) \rightarrow \Rhd_K(X)$.
% then clearly $(K\times_L i) (K\times_L X)\subseteq K\times_L Y_X$, and since $K\times_L -$ preserves retractions we get an induced retraction 
%\[ \xymatrix{
%K \times_L Y_X \ar@/^/[d]^-r \ar[r] & q_!Y_X \ar@/^/[d]^-r \\
%K \times_L X \ar[r] \ar[u] & X. \ar[u] } \] 
%
% \noindent Thus $q_!Y$ is of the form $q_!Y_X  \sqcup \bigsqcup (D_i \times L/L_\alpha)$, with $q_! Y_X$ as $K$-equivariant retract.
 \end{proof}

The above constructions are all functorial in $X$. As for ordinary $A$-theory, this requires a careful definition, see for instance \cite[Rem 3.5]{raptis_steimle}.
\begin{df}\label{functoriality_definition} 
	Fix a set $U$ of size $2^{|\R|}$ and declare that, by definition, each retractive space $Y$ over $X$ has as its underlying set $X \sqcup S$ where $S \subseteq U$. To each $G$-equivariant map $f\colon X \to X'$ of base spaces, define the pushforward functor $f_!\colon R(X) \to R(X')$ by taking each space $Y = X \sqcup S$ to the pushout $Y \cup_X X'$, with underlying set $X' \sqcup S$, and projection to $X'$ by $S \to X \overset{f}\to X'$.
\end{df}
%\mmar{I know RS use $f_{!}$ but maybe to avoid the conflict of  notation with our transfer map notation, we could change it. What about $f_\ast$ for example?}\cmar{I just tried that, I don't like it because in parametrized homotopy theory $f_*$ means something else}\mnote{what about $f_\%$? That's also sometimes used for something else, but it might still be better to pick one of the names that have conflicts not in this paper and just make a remark that usually it stands for bla but that that doesn't appear in this paper. Or if you prefer the $f_!$, we should at least make a remark about that that it is different than the pushforward for the spans, but it's clear from the context which one we mean}
\begin{lem}
	This construction is strictly functorial, i.e. $(f' \circ f)_! = f'_!f_!$, and strictly $G$-equivariant, i.e. $g \circ f_! = f_! \circ g$ for all $g \in G$.
\end{lem}
As a result the homotopy fixed point categories $R(X)^{hH}$ are also functorial in $X$, by composing each $G$-equivariant functor $F\colon \tG \times G/H \to R(X)$ with the $G$-equivariant functor $f_!\colon R(X) \to R(X')$.
\begin{rem}
	In this paper the notation $(-)_!$ refers to two different kinds of pushforward functors, a pushforward $p_!$ on fixed point categories $R(X)^{hH}$ when $p$ is a map of finite $G$-sets $G/H \to G/K$, and a pushforward $f_!\colon R(X)\to R(X')$ when $f$ is a map of $G$-spaces $X \to X'$. The meaning should be clear from which map is being shrieked, but we will reinforce this distinction by only using $f$ to refer to maps of base spaces, never to maps of finite $G$-sets.
\end{rem}
For each map of finite $G$-sets $p\colon S \to T$, the operation $f_!$ strictly commutes with $p^*$, and it strictly commutes with $p_!$ as soon as we fix for each $t \in T$ an injection $\bigsqcup_{p^{-1}(t)} U \to U$, so that we know how to embed disjoint unions along $X$ into $X \sqcup U$. It furthermore preserves maps between spans, so that the following square of categories strictly commutes.
\[\xymatrix @R=2em @C=3em{
	\mc S_{H,K} \times R(X)^{hK}  \ar[d]_-{\id\times f_!} \ar[r]   \drtwocell<\omit>{<0> \ \ =}  & R(X)^{hH} \ar[d]^-{f_!} \\
	\mc S_{H,K} \times R(X')^{hK}  \ar[r]& R(X')^{hH}
}\]
We define $f_!$ on the strictifications $\underline{R(X)^{hH}}$ by sending the object $(J,F,S)$ to $(J,f_! \circ F,S)$, and defining it on morphisms so that following square commutes.
\[\xymatrix @R=2em @C=3em{
	\underline{R(X)^{hH}} \ar[d]_-\sim \ar@{-->}[r]^-{f_!} & \underline{R(X')^{hH}} \ar[d]^-\sim \\
	R(X)^{hH} \ar[r]^-{f_!} & R(X')^{hH}
}\]
%\[\xymatrix @R=2em @C=3em{
%	S_{H,K} \times R(X)^{hK}  \ar[d]_-{\id\times f_!} \ar[r]   \drtwocell<\omit>{<0> \ \ =}  & R(X)^{hH} \ar[d]^-{f_!} \\
%	S_{H,K} \times R(X')^{hK}  \ar[r]& R(X')^{hH}
%}
%\qquad
%\xymatrix @R=1.85em @C=3em{
%	S_{H,K} \times \underline{R(X)^{hK}}  \ar[d]_-{\id\times f_!} \ar[r]   \drtwocell<\omit>{<0> \ \ =}  & \underline{R(X)^{hH}} \ar[d]^-{f_!} \\
%	S_{H,K} \times \underline{R(X')^{hK}}  \ar[r]& \underline{R(X')^{hH}}
%}\]
We then check that the square
\[\xymatrix @R=2em @C=3em{
	\mc S_{H,K} \times \underline{R(X)^{hK}}  \ar[d]_-{\id\times f_!} \ar[r]   \drtwocell<\omit>{<0> \ \ =}  & \underline{R(X)^{hH}} \ar[d]^-{f_!} \\
	\mc S_{H,K} \times \underline{R(X')^{hK}}  \ar[r]& \underline{R(X')^{hH}}
}\]
commutes. On objects this is immediate. On morphisms in the $\underline{R(X)^{hK}}$ direction, as in part 7 of the proof of \cite[Prop 4.11]{CaryMona}, it boils down to the commutativity of
\begin{equation}\label{BC_commutes_with_functoriality}
\xymatrix{
	f_! \circ (S * T) * F \ar@{=}[d] \ar[r]^-{\textup{BC}}_-\cong & f_! \circ (S * (T * F)) \ar@{=}[d] \\
	(S * T) * (f_! \circ F) \ar[r]^-{\textup{BC}}_-\cong & S * (T * (f_! \circ F)),
}
\end{equation}
and on morphisms in the $S_{H,K}$ direction it similarly boils down to checking that $f_!$ commutes with the map $(S * T) * F \to (S' * T) * F$ for a map of spans $S \to S'$. In both cases the equalities follow because $f_!$ only changes the basepoint section, and doesn't interfere with the summands that the other operations are re-arranging.
%This is verified by a careful check of definitions; a slightly more elegant approach shows that the equality $f_!q^* = q^*f_!$ is an instance of a Beck-Chevalley map, and then the above square follows from the standard fact that Beck-Chevalley maps are preserved by pasting, see for instance \cite[Prop 8.7, (*!!)]{mp2}\cmar{"Coherence for indexed symmetric monoidal categories"}.
This finishes the proof of the following proposition.
\begin{prop}\label{all_natural_in_X}
	Each of the constructions from this section defines a spectral Mackey functor that is functorial in the $X$ variable.
\end{prop}

\subsection{$\Sigma^\infty_G X_+$ as a spectral Mackey functor} We have constructed a map of spectral Mackey functors
\[ \left\{ K^\Wald\left(\underline{R(X)^{hH}_{h\delta}}\right) \right\} \rightarrow \left\{ K^\Wald\left(\underline{R(X)^{hH}_{f}}\right) \right\}, \]
inducing a map from some $G$-spectrum into $\bA_G(X)$. So it remains to prove that this $G$-spectrum is $\Sigma^\infty_G X_+$:
\begin{thm}\label{it_models_suspension_spectrum}
	The $G$-spectrum modeled by $\left\{ K^\Wald\left(\underline{R(X)^{hH}_{h\delta}}\right) \right\}$ is naturally equivalent to $\Sigma^\infty_G X_+$, as a functor from $G$-spaces to $G$-spectra.
\end{thm}

The proof occupies the rest of the section. We first reduce the statement to the case when $X$ is a finite $G$-set, essentially by proving that both functors are enriched homotopy left Kan extensions from the subcategory of $G$-orbits $G/H$ to all $G$-spaces.

To make sense of this, we first have to make the two functors simplicially enriched. We only know that they are homotopy functors (i.e. they preserve equivalences), but there is a standard trick to convert homotopy functors into simplicial functors, going back at least as far as \cite[Lem 3.1.2]{waldhausen}. The argument is given in detail in \cite[Lem 3.4]{malkiewich_merling_coassembly} for contravariant functors. The argument for covariant functors is essentially the same but uses the realization $|[n] \mapsto F(\Map_{sSet}(\Delta^n, X))|$ instead of $|[n] \mapsto F(\Delta^n \times X)|$. The fact that $X$ is a $G$-space only adds a few additional points to the proof, checking that the maps we feed into the functor $F$ are indeed equivariant maps. This proves the following.

\begin{lem}\label{replace_enriched}
	Let $F$ be a homotopy functor from unbased $G$-spaces to $G$-spectra. Then $F$ can be replaced up to a zig-zag of equivalences of functors, functorial in $F$, by a simplicially enriched functor.
\end{lem}

%\mnote{ I am actually confused now about how to do this. 
%	
%	Waldhausen defines a new functor 
%	$$\tilde{F}(X)=| q\mapsto F(Map_{sset}(\Delta^n, X))|$$
%	This doesn't directly apply for us because his $F$ takes values in sset, and the result is valued in simplicial objects in $\mathcal{B}$, the target category of his $F$, and then $F$ and $\tilde{F}$ are equivalent when $F$ is viewed as valued in constant simplicial objects in $\mathcal{B}$, when $F$ is a homotopy functor.
%	
%	Our functor takes values in $G$-spaces, and our target $\mathcal{B}$ is $G$-spectra. So maybe we would want to do something like what we did in the coassembly paper and define a replacement 
%	$$\tilde{F}(X)=( q\mapsto F(|Map_{sset}(\Delta^n, Sing X)|)$$
%	which would be valued in simplicial $G$-spectra.
%	
%	But it is not clear to me that this is what we want to do, and whether we actually would need an equivariant version of Sing, taking into account simplices of the form $G/H\times \Delta^n$....
%}

Next we define an equivariant assembly map for any enriched functor $F$ from $G$-spaces to $G$-spectra; for a comprehensive discussion of equivariant assembly maps, see \cite{DavisLuck}. Let $X_{\sbt}$ be any $G$-simplicial set -- as for $G$-spaces, an equivalence of $G$-simplicial sets is a map that is an equivalence on the $H$-fixed points $|X_{\sbt}|^H \cong |X_{\sbt}^H|$ for all $H \leq G$. Let $\sO(G)_+$ denote the orbit category with objects $G/H$, and let $F$ be any simplicial functor from unbased $G$-spaces to $G$-spectra.

%Let $\sO(G)_+$ denote the orbit category with objects $G/H$, which is enriched in based simplicial sets by adding disjoint basepoints to the discrete mapping sets $\sO(G/H, G/K)_+$.
Define the bar construction $B(F(G/-),\sO(G)_+,|X_{\sbt}|^{-}_+)$ as the realization of the simplicial $G$-spectrum
$$q\mapsto \bigvee_{H_1,\dots, H_q } F(G/H_q)\sma \sO_G(G/{H_q}, G/H_{q-1})_+\sma \dots \sma \sO_G(G/H_2, G/H_1)_+\sma |X_{\sbt}|^{H_1}_+,$$
where $F(G/H_q)$ is a $G$-spectrum and the other spaces in the smash product have trivial $G$-action.
Note that $|X_{\sbt}|^{H_1} \cong |X_{\sbt}^{H_1}| \cong |\Map_\mathrm{sSet}(G/H_1, X_{\sbt})|$, where $G/H_1$ is viewed as a constant simplicial set and so the simplicial enrichment of $F$ provides a map from the above expression to $F(|X_{\sbt}|)$, compatible with the faces and degeneracies.
This defines the \emph{equivariant assembly map}
\begin{equation}\label{equivariant_assembly}
B(F(G/-),\sO(G)_+,|X_{\sbt}|^{-}_+) \longrightarrow F(|X_{\sbt}|).
\end{equation}
Though we will not formalize this, $F$ is an enriched homotopy left Kan extension from the subcategory of $G$-orbits precisely when \eqref{equivariant_assembly} an equivalence. Note also that this %condition
definition makes sense for any homotopy functor $F$, in light of \autoref{replace_enriched} and the fact that the assembly map respects weak equivalences in the $F$ variable. This condition is really a variant of the ``excisive'' condition from Goodwillie calculus \cite{calc1}, only the source category is $G$-spaces instead of spaces, so the one-point space is replaced by the orbits $G/H$ for all $H$.

%\mnote{Alright, last bubble note, this is again just for my own conscience. If it's correct, delete it (i.e comment it out). We want to see how $F$ being simplicially enriched gives a map of the form
%$F(|X.|)\sma |\Map_{sset}(X.,Y.)|\to F(|Y.|)$. I realized it's actually not hard to fill in the details of what you said in the chat. What it means for $F$ to be simplicial is that we have a map of ssets
%$$Sing \Map(|X.|,|Y.|) \to Sing \Map(F(|X.|),F(|Y.|)), \text{so a map} $$ 
%$$|Sing \Map(|X.|,|Y.|)| \to  \Map(F(|X.|),F(|Y.|)).$$
%Now note that there is a map $\Map_{sset}(X.,Y.) \to Sing \Map(|X.|,|Y.|)$, thus precomposing with this we get the map $$|\Map_{sset}(X.,Y.)|\to \Map(F(|X.|),F(|Y.|))$$ whose adjoint we want to take.
%}

\begin{rem}
Note that since $F$ is a homotopy functor, for any $G$-space $X$, we have an equivalence $F(|\mathrm{Sing}_{\sbt}(X)|)\to F(X)$, thus at the cost of  defining the assembly map via a  zig-zag, we could define it for all $G$-spaces, not only realizations of simplicial $G$-sets.	Alternatively, we could also define the equivariant assembly map using categories of simplices, which eliminates the need to make $F$ simplicially enriched, but the proofs become longer. However, defining the assembly map on $G$-spaces of the form $|X_{\sbt}|$ will be  enough for us to draw the conclusion that we want.
\end{rem}
%\mnote{Or what about instead of this remark we do the following: we are trying to show that two functors $F_1, F_2$ are equivalent on all $G$-spaces $X$. The assembly map argument show they are equivalent on all $G$-spaces $|X.|$ (if they agree on fixed points). Now we could simply use the diagram to conclude the right hand map is also an equivalence (sorry apparently diagrams don't go in the bubbles)}
%$$\xymatrix{
%F_1(|\mathrm{Sing}_{\sbt}(X)|)\ar[d] \ar[r] & F_1(X)\\
%F_2(|\mathrm{Sing}_{\sbt}(X)|) \ar[r] &F_2(X)
%}$$ 
%\mnote{actually, question: do we a priori have a map from our spectral Mackey functor to the suspension functor}
%\cnote{No we don't get a map a priori, so remove that map from the diagram (I already did) and observe that you still get a zig-zag of weak equivalences from $F_1$ to $F_2$, so we still conclude they're equivalent.}

\begin{prop}\label{fixedpointsagree}
Let $F$ be a functor from $G$-spaces to $G$-spectra with the property that it agrees on $H$-fixed points with  $(\Sigma^\infty_G X_+)^H$ for every subgroup $H$. Then the equivariant assembly map \eqref{equivariant_assembly} is a $G$-equivalence of $G$-spectra for every $G$-space of the form $|X_{\sbt}|$.
\end{prop}	

\begin{proof}
	First note that a functor as in the proposition is a homotopy functor, since equivalences in $G$-spectra are determined on fixed points. It therefore has an equivariant assembly map.
	
	We show that for any subgroup $H$, the nonequivariant map of genuine fixed point spectra 
	$$B(F(G/-),\sO(G)_+,|X_{\sbt}|^{-}_+)^H \longrightarrow F(|X_{\sbt}|)^H$$ is an equivalence. Since taking genuine fixed points commutes up to equivalence with geometric realizations, coproducts, and smashing with a space with trivial $G$-action, this reduces to proving that the map
	$$B(F(G/-)^H,\sO(G)_+,|X_{\sbt}|^{-}_+) \longrightarrow F(|X_{\sbt}|)^H$$ is an equivalence. Since this only uses the $H$-fixed point information of $F(X)$, it is therefore enough to prove this statement for $\Sigma^\infty_{G} (-)_+$. Pulling out the suspension spectrum from both sides of \eqref{equivariant_assembly}, the assembly map for $\Sigma^\infty_{G} (-)_+$ becomes the equivalence in the statement of Elmendorf's theorem for $G$-spaces; see also \cite{DavisLuck}.
%	\mnote{OK, I see my confusion before. This is just the following statement: if $X\to Y$ is an equivalence of $G$-spaces then $\Sigma_G^\infty X_+\to\Sigma_G^\infty Y_+$ is an equivalence of $G$-spectra and we can think of $\Sigma_G^\infty X_+$ as $S_G\sma X_+$ to pull out $S_G\sma$ out of the wedge, is that right?
%	}
%	\cnote{Exactly!}
%	\mnote{Actually, I am seeing a problem here with what we said before, which was ``Pulling out the suspension spectrum from both sides of \eqref{equivariant_assembly}, its assembly map becomes the equivalence in the statement of Elmendorf's theorem for $G$-spaces". We are now looking at the map 
%		$$|q\mapsto \bigvee_{H_1,\dots, H_q } (\Sigma^\infty_G (G/H_q)_+)^H \sma \sO_G(G/{H_q}, G/H_{q-1})_+\sma \dots \sma O_G(G/H_2, G/H_1)_+\sma X^{H_1}_+| \to (\Sigma^\infty_GX_+)^H$$ But suspension $G$-spectrum doesn't commute with fixed points, so can't simply pull out the $ \mathbb{S} \sma$ to pull it out...
%	}
\end{proof}

By \cite[Theorem 1.1.]{BD}, $K^\Wald\left(R(X)^{hH}_{h\delta}\right)$ is equivalent to $(\Sigma^\infty_G X_+)^H$, and by \autoref{inclusions}    
$\left\{K^\Wald\left(R(X)^{hH}_{h\delta}\right)\right\}\simeq \left\{ K^\Wald\left(\underline{R(X)^{hH}_{h\delta}}\right) \right\}$ assemble into a spectral Mackey functor. Thus the functor from $G$-spaces to $G$-spectra determined by sending a $G$-space $X$ to the $G$-spectrum modeled by this spectral Mackey functor satisfies the hypothesis in \autoref{fixedpointsagree}, therefore its equivariant assembly map is an equivalence for $G$-spaces of the form $|X_{\sbt}|$. Thus when $F$ is either our spectral Mackey functor or the suspension $G$-spectrum, for any $G$-space $X$, we get equivalences of $G$-spectra
$$B(F(G/-),\sO(G)_+,|\mathrm{Sing}_{\sbt}(X)|^{-}_+) \xrightarrow{\simeq} F(|\mathrm{Sing}_{\sbt}(X)|)\xrightarrow{\simeq} F(X).$$

Therefore each of these functors is equivalent to the source of its assembly map, which only depends on the behavior of the functor on the orbit category $\sO(G)$. We have therefore reduced the proof of \autoref{it_models_suspension_spectrum} to showing that the two functors are equivalent when restricted to finite $G$-sets.

It remains to prove that  our spectral Mackey functor 
 $\left\{ K^\Wald\left(\underline{R(X)^{hH}_{h\delta}}\right) \right\}$ is naturally equivalent to the suspension spectrum $\Sigma^\infty_G X_+$, as $X$ ranges over all finite $G$-sets. Under this restriction, we can simplify our spectral Mackey functor as follows:
\[ \xymatrix{
	\left\{ K^\Seg\left(i\Lambda\underline{R(X)^{hH}_{\delta}}\right) \right\} \ar[r]^-\sim &
	\left\{ K^\Wald\left(\underline{R(X)^{hH}_{\delta}}\right) \right\} \ar[r]^-\sim &
	\left\{ K^\Wald\left(\underline{R(X)^{hH}_{h\delta}}\right) \right\}
} \]
The left-hand map of Mackey functors is the map from \autoref{translate_modules} over the equivalence $\GB \xrightarrow{\simeq} \GB^\Wald$. The following lemma is immediate from the definitions, and by \autoref{segal_equals_waldhausen} implies this is an equivalence of Mackey functors (i.e. gives an equivalence on $K$-theory spectra at each $H \leq G$).
\begin{lem}\label{cofiber_sequences_split}
	The cofiber sequences in $\Rd_H(X)$ split in a way that is functorial along weak equivalences, as in \autoref{segal_equals_waldhausen}.
\end{lem}
The right-hand map of Mackey functors is an equivalence by the following lemma.
\begin{lem}\label{d_and_hd_equivalent}
	When $X$ is discrete, i.e. a $G$-set, the inclusion $\Rd_G(X) \to \Rhd_G(X)$ induces an equivalence on Waldhausen $K$-theory.
\end{lem}

\begin{proof}
	We define an exact functor the other way by sending each retractive space $Y$ over $X$ to its set of connected components, $X \sqcup \pi_0(Y \setminus Y_X)$. (If $X$ is not discrete, this does not produce a space over $X$ in a natural way.) The composite $\Rd \to \Rhd \to \Rd$ is naturally isomorphic to the identity functor, while $\Rhd \to \Rd \to \Rhd$ is naturally weakly equivalent to the identity, along the map $Y \to X \sqcup \pi_0(Y \setminus Y_X)$ that projects $Y_X$ to $X$ and collapses each disc to a point.
\end{proof}

\begin{rem}
	The above lemma is consistent with the fact that when $X$ is discrete, both $K(\Rd(X))$ and $K(\Rhd(X))$ are equivalent to $\Sigma^\infty_+ X$. For more general spaces $X$, $K(\Rhd(X))$ is still the suspension spectrum of $X$, but $K(\Rd(X))$ is only the suspension spectrum of its underlying set.
\end{rem}
%
%Our final goal in this section is to prove the same for the spectral Mackey functor $\{K(\Rhd_H(X)\}$ -- namely, that it models $\Sigma^\infty_G X_+$ when $X$ is a $G$-space. From \cite{BD} we know that each individual spectrum $K(\Rhd_H(X))$ agrees with the \emph{$H$-fixed points} $(\Sigma^\infty_G X_+)^H$, and when $X$ is discrete the same is true of $K(\Rd_G(X))$. \mona{(BD mention this is a short elaboration of Rognes's proof for $X=\ast$)}.
%
%Recall from \cite{BD} that on the $G$-fixed points, this map respects the tom Dieck splittings (and similarly for the $H$-fixed points when $H \leq G$). It therefore remains to prove that the $G$-spectrum associated to the $\GB^\Wald$-module $\{K^\Wald(\underline\Rhd_H(X))\}$ does not just have $(\Sigma^\infty_{+G} X)^H$ as its fixed points for all $H \leq G$, but it is actually $\Sigma^\infty_{+G} X$. (Of course this will not be surprising, but it does not logically follow from what we know so far.)

We are left with a Mackey functor of symmetric monoidal categories over $\GE$, which can be compared directly to the Mackey functor from the following result.
%We use the following theorem that states that the suspension $G$-spectrum is given by the representable $\GE$-module.
\begin{thm}[\cite{BO}, Thm. 9.4.]
	Let $X$ be a finite $G$-set. The $\GB$-module given by $K^\Seg(G\sE(-, X))$ represents the suspension $G$-spectrum $\Sigma_G^\infty X_+$.
\end{thm}
%Therefore, to form an equivalence to $\Sigma^\infty_{G} (-)_+$, it remains to prove t

%
%\mmar{this paragraph will now need editing}
%Now it remains to carry over this result from \cite{BO}, which is stated using permutative categories and Segal $K$-theory. To accomplish the translation, we first discuss how to pass between Waldhausen $K$-theory and Segal $K$-theory while preserving the multiplicative structure that makes $\GB^\Wald$ into a spectral category and $\{K(\underline\Rd_H(X)\}$ into a module over it. To encode this multiplicative structure we use the language of multicategories, see \cite{BOmult} for more details.

%%%%%%%%THIS IS WHERE A BIG CHUNK WAS CUT AND MOVED 

\subsection{Reconciling the different models for $\Sigma_G^\infty X$}
We have now reduced the problem to giving a natural map of $\GE$-modules in cocartesian monoidal categories
$$\left\{ i\Lambda\underline{R(X)^{hH}_\delta} \right\}_{H \leq G} \longrightarrow \left\{G\sE(G/H, X)\right\}_{H \leq G}$$
that for each $H$ is an equivalence of categories, and therefore gives an equivalence of Segal $K$-theory spectra. By \autoref{pseudo_can_be_strictified}, it is enough to give a pseudo linear map of modules. By \autoref{thickening_preserves_functoriality} and \autoref{natural_thickens_to_natural}, to make this natural in $X$ it is enough to make the pseudo linear map natural in the sense of \autoref{pseudo_is_natural}.

We begin by concretely describing the map up to equivalence of categories. Recall that the category $i\Lambda \big( \Rd_H(X)\big)$ consists of retractive $H$-spaces over $X$ of the form $X\rightarrow X\sqcup S\rightarrow X$, with isomorphisms between them and symmetric monoidal structure given by sum under $X$. On the other hand $G\sE(G/H,X)$ consists of finite $G$-sets over $G/H\times X$, with disjoint union. We map
$$i\Lambda \big( \Rd_H(X)\big) \longrightarrow G\sE(G/H, X)$$
by sending $X \sqcup S$ with projection $p$ to the set $G \times_H S$, with each point $(g,s) \in G \times_H S$ projecting to $gH \in G/H$ and $gp(s) \in X$. This operation is an equivalence of categories and therefore preserves the coproduct. The inverse functor takes each span $G/H \xleftarrow{\rho}  S\xrightarrow{p} X$ to $X \sqcup \rho^{-1}(eH)$, with the map to $X$ provided by $p$.

\begin{prop}
	Let $X$ be a finite $G$-set. There is a pseudo linear map (\autoref{pseudolinear}) of $\GE$-modules in $\cCMC$
	$$\{ i\Lambda\underline{R(X)^{hH}_\delta}\} \longrightarrow \{G\sE(G/H, X)\}$$
	that is natural in $X$, and up to equivalence agrees with the equivalence of categories described just above.
\end{prop}

\begin{proof}
Since we are only considering isomorphisms between discrete retractive spaces, we can ignore the basepoint section. So when we write $F(g_1,g_2 K)$, we mean just the disjoint set and not the copy of $X$. Combining the above discussion with \autoref{functor_to_space} and \autoref{Rdhtpyfixedpts} gives an equivalence of categories
$$f_K\colon i\Lambda R(X)^{hK}_\delta \overset{\textup{\autoref{Rdhtpyfixedpts}}}\simeq i\Lambda(\Rd(X))^{hK}\overset{\textup{\autoref{functor_to_space}}}\simeq i\Lambda\Rd_K(X) \longrightarrow G\sE(G/K, X)$$
taking each functor $F\colon \tG \times G/K \to \Rd(X)$ to the $G$-set over $G/K \times X$ given by
\[ \xymatrix @R=1em @C=1em{
	& G \times_K F(e,K) \ar[ld] \ar[rd] & \\
	G/K && X,
} \]
with $k \in K$ acting on $F(e,K)$ by
\[ \xymatrix @C=4em{ F(e,K) \ar@{=}[r] & F(k,K) \ar[r]^-{F(k \to e,K)} & F(e,K). } \]

	The fact that $F$ is $G$-equivariant implies that the square below commutes. This demonstrates that the action of $k$ on $F(e,K)$ can equally well be understood as the action of any arrow in $\tG$ that ``right multiplies by $k^{-1}$.''
	\[ \xymatrix @C=5em{ F(e,K) \ar@{=}[r] \ar@{=}[dr] & F(k,K) \ar@{=}[d] \ar[r]^-{F(k \to e,K)} & F(e,K) \ar@{=}[d] \\
		& F(g,gK) \ar[r]^-{F(g \to gk^{-1},K)} & F(gk^{-1},gK). } \]
	
Next we define natural isomorphisms
\begin{equation}\label{theta_HK}
\xymatrix @R=2em @C=3em{
%		S_{K,H}\times G\sE(G/H, X)  \ar[d]_-{\id\times f_H} \ar[r]   \drtwocell<\omit>{<0> \ \ \ \   \theta_{K,H}}  & G\sE(G/K, X) \ar[d]^-{f_K} \\%
%		S_{K,H}\times i\Lambda\underline\Rd_H(X)  \arrow[r]& i\Lambda\underline\Rd_K(X).
		G\sE(G/H,G/K) \times i\Lambda R(X)^{hK}_\delta  \ar[d]_-{\id\times f_K} \ar[r]   \drtwocell<\omit>{<0> \ \ \ \   \theta_{H,K}}  & i\Lambda R(X)^{hH}_\delta \ar[d]^-{f_H} \\%
		G\sE(G/H,G/K) \times G\sE(G/K, X)  \arrow[r]& G\sE(G/H, X).
	}
\end{equation}
Let $G/H \xleftarrow{p}  S\xrightarrow{q} G/K$ be a span in $G\sE(G/H,G/K)$ and $\tG \times G/K \xrightarrow{F} \Rd(X)$ be a functor in the top-left category. The bottom-left route takes $(S,F)$ to the span
	\[ \xymatrix @R=1em @C=1em{
		& S \times_{G/K} \left( G \times_K F(e,K) \right) \ar[ld] \ar[rd] & \\
		G/H && X } \]
	and the top-right route takes it to the span
	\[ \xymatrix @R=1em @C=1em{
		& G \times_H \left( \coprod_{i \in p^{-1}(H)} F(e,q(i)) \right) \ar[ld] \ar[rd] & \\
		G/H && X. } \]
	so we just need a natural isomorphism between these two spans.
	
	We first explain how we calculated these in a little more detail. The first is essentially by definition, since $f_K$ produces $G \times_K F(e,K)$ and $S$ acts by pulling back from $G/K$ to $S$. For the second, $S$ first acts on $F$, producing the functor $S * F = p_!q^*F\colon \tG \times G/H \rightarrow \Rd(X)$. We recall from \cite[Def 4.7]{CaryMona} the formula
	\[ (S*F)(g_1,g_2H) = (p_!q^*F)(g_1,g_2H) = g_1\left( \coprod_{i \in p^{-1}(g_1^{-1}g_2H)} F(e,q(i)) \right) \cong \coprod_{j \in p^{-1}(g_2H)} F(g_1,q(j)), \]
	The above isomorphism is chosen to be the canonical map that commutes the functor $g_1$ with the coproduct, plus the equality
	\[ g_1F(e,q(i)) = F(g_1,g_1q(i)) = F(g_1,q(j)) \textup{ for } j = g_1i. \]
	This isomorphism is used to define the adjunction between $p_!$ and $p^*$ on the categories of $G$-equivariant functors, and to define the action of $p_!q^*F$ on morphisms in $\tG \times G/H$. Each morphism $(g_1 \rightarrow g_3,g_2H)$ is sent to the coproduct morphism
	\[ \xymatrix @C=7em {
		g_1\left( \coprod_{i \in p^{-1}(g_1^{-1}g_2H)} F(e,q(i)) \right) \ar@{<->}[d]_-\cong & g_3\left( \coprod_{i \in p^{-1}(g_3^{-1}g_2H)} F(e,q(i)) \right) \ar@{<->}[d]^-\cong \\
		\coprod_{j \in p^{-1}(g_2H)} F(g_1,q(j)) \ar[r]^-{\coprod F(g_1 \rightarrow g_3,q(j))} & \coprod_{j \in p^{-1}(g_2H)} F(g_3,q(j)).
	} \]
	With these choices $p_!$ is indeed the equivariant left adjoint of $p^*$, well-defined up to a canonical equivariant isomorphism that arises either from the uniqueness of left adjoints or through uniqueness of the coproduct (both give the same map).
%	\cnote{I just thought about this ridiculously carefully. If you adopt the point of view that $p_!$ should be defined as the left adjoint of $p^*$, it really does lead to the coproduct and every map is the obvious coproduct of maps you started with.}
	Using the fact that the $h$-action on $R(X)$ preserves total spaces and the morphisms between them, we compute that the $h$-action on
	\[ (p_!q^*F)(e,H) = \coprod_{i \in p^{-1}(H)} F(e,q(i)) \]
	is the map that sends the $i$ summand to the $hi$-summand by
	\[ \xymatrix @C=6em{ F(e,q(i)) \ar@{=}[r] & F(h,q(hi)) \ar[r]^-{F(h \to e,q(hi))} & F(e,q(hi)). } \]
%	\cary{Key step: since $h$ acts trivially on total spaces and morphisms in $R(X)$, the following square commutes:
%		\[ \xymatrix{ \coprod_{i \in p^{-1}(H)} F(e,q(i)) \ar@{=}[r] & h\coprod_{i \in p^{-1}(H)} F(e,q(i)) \\
%			F(e,q(i)) \ar[u] \ar@{=}[r] & F(h,q(hi)) \ar[u]
%		} \]
%	}
	As a consistency check, we see directly that the map $\coprod_{i \in p^{-1}(H)} F(e,q(i))\to X$ is $H$-equivariant. At any rate, this is the $H$-action that we use when we apply $f_H$.
%	 and get the span
%	\[ \xymatrix @R=1em @C=1em{
%		& G \times_H \left( \coprod_{i \in p^{-1}(H)} F(e,q(i)) \right) \ar[ld] \ar[rd] & \\
%		G/H && X. } \]
	
	Now we define the isomorphism %of spans over $G/H$ and $X$
	\[ \xymatrix{ S \times_{G/K} \left( G \times_K F(e,K) \right) & \ar@{<->}[l]_-{\theta_{H,K}}^-\cong G \times_H \left( \coprod_{i \in p^{-1}(H)} F(e,q(i)) \right). } \]
	Heuristically, each one is a coproduct over $S$ of various values of $F$, so the trick is to pick the right identification between these values. We assign a triple $(s,g,x) \in S \times G \times F(e,K)$ to the triple $(\gamma,i,y) \in G \times \coprod_{i \in p^{-1}(H)} F(e,q(i))$ by taking $\gamma$ to be a lift of $p(s)$ to $G$:
	\[ \xymatrix @R=0.5em @C=1em{
		G \ar[rd] && S \ar[ld]_-p \\
		\gamma \ar[rd] & G/H & s \ar[ld] \\
		& \gamma H = p(s) & } \]
	Then take $i = \gamma^{-1}s$, and $y$ to be the image of $x$ under any of the maps in the commuting diagram
	\begin{equation}\label{eq:critical_diagram}
		\xymatrix @C=6em @R=1em{
		x \in F(e,K) \ar@{=}[d] \ar[r]^-{F(e \to g^{-1}\gamma,K)} & F(g^{-1}\gamma,K) \ar@{=}[d] \\
		F(g,gK) \ar@{=}[d] \ar[r]^-{F(g \to \gamma,gK)} & F(\gamma,gK) \ar@{=}[d] \\
		F(g,\gamma q(i)) \ar@{=}[d] \ar[r]^-{F(g \to \gamma,\gamma q(i))} & F(\gamma,\gamma q(i)) \ar@{=}[d] \\
		F(\gamma^{-1}g,q(i)) \ar[r]^-{F(\gamma^{-1} g \to e,q(i))} & F(e,q(i)) \ni y.
	}
	\end{equation}
	
	This diagram uses the equality $gK = q(s) = \gamma q(i)$. The following diagram is helpful for quickly checking this and several other equalities in this proof.
	\[ \xymatrix @R=0.5em @C=1em{
		G \ar[rd] && S \ar[ld]_-p \ar[rd]^-q && G \ar[ld] \\
		\gamma \ar[rd] & G/H & \gamma i = s \ar[ld] \ar[rd] & G/K & g \ar[ld] \\
		& \gamma H = p(s) && \gamma q(i) = q(s) = gK } \]
	
	If we replace $\gamma$ by $\gamma h$, we instead get the triple $(\gamma h, h^{-1}i, h^{-1}y)$, by appending the following commuting diagram to the bottom of \eqref{eq:critical_diagram}:
	\[ \xymatrix @C=10em @R=2em{
		F(\gamma^{-1}g,\ q(i)) \ar@{=}[d] \ar[r]^-{F(\gamma^{-1} g \to e,\ q(i))} & F(e,\ q(i)) \ni y \ar@{=}[d] \\
		F(h^{-1}\gamma^{-1}g,\ q(h^{-1}i)) \ar[dr]_-{F(h^{-1}\gamma^{-1} g \to e,\ q(h^{-1}i))\ \ \ \ \ \ \ \ \ } \ar[r]^-{F(h^{-1}\gamma^{-1} g \to h^{-1},\ q(h^{-1}i))} & F(h^{-1},\ q(h^{-1}i))  \ar[d]^-{F(h^{-1} \to e,\ q(h^{-1}i))} \\
		& F(e,\ q(h^{-1}i)) \ni h^{-1}y
	} \]
	This gives the same point in the quotient $G \times_H \left( \coprod_{i \in p^{-1}(H)} F(e,q(i)) \right)$. The other choice in this construction was the choice of representative $(s,g,x)$ of a point in the domain. Any other representative takes the form $(s,gk,k^{-1}x)$, but this results in the same output, by appending the following diagram to the top of \eqref{eq:critical_diagram}.
	\[ \xymatrix @C=10em{
		F(k^{-1},\ K) \ar@{=}[d] \ar[r]^-{F(k^{-1} \to k^{-1}g^{-1}\gamma,\ K)} & F(k^{-1}g^{-1}\gamma,\ K) \ar@{=}[d] \\
		x \in F(e,K) \ar[ur]^-{F(e \to k^{-1}g^{-1}\gamma,\ K)\ \ \ \ \ \ } \ar[r]^-{F(e \to g^{-1}\gamma,\ K)} & F(g^{-1}\gamma,\ K) \\
	} \]
	The outside square of this commutes, and the upper-left triangle when followed clockwise brings $x$ to $k^{-1}x$. Therefore the bottom map on $x$ (the original definition) agrees with the diagonal map on $k^{-1}x$, followed by the right vertical (the new definition).
	
	So we have a well-defined map of sets. It commutes with the projection to $G/H$ because both triples are sent to $p(s)=\gamma H$. We check compatibility with the map to $X$. Recall that $q(i) = \gamma^{-1}gK$. Let $\rho\colon F(e, K)\to X$ and $\rho'\colon F(e, q(i))\to X$ be the projections. Recall that although $F(\gamma^{-1}g,\gamma^{-1}gK)$ is equal to $F(e,K)$ as a set, it has a different projection to $X$, namely $\gamma^{-1}g \circ \rho$. Since $F$ sends each morphism in $\tG \times G/K$ to a map of retractive spaces, we get a commuting diagram
	\[\xymatrix @C=3em @R=1em{
		F(\gamma^{-1}g,q(i)) \ar[dr]_-{\gamma^{-1}g \circ \rho} \ar[rr]^-{F(\gamma^{-1} g \to e,q(i))} && F(e,q(i))\ar[dl]^-{\rho'} \\
		& X& 
	}\]
	Combining with \autoref{eq:critical_diagram}, we conclude $\gamma^{-1}g\rho(x)=\rho'(y)$, or $g\rho(x)=\gamma\rho'(y)$. But the triple $(s,g,x)$ projects to $g\rho(x)$ and the triple $(\gamma,i,y)$ projects to $\gamma\rho'(y)$, so our correspondence does indeed respect the projection to $X$.

	The map we constructed above has an inverse that sends $(\gamma,i,y)$ to $(s,g,x)$ by taking $s = \gamma i$, $g$ a lift of $\gamma q(i) \in G/H$ to $G$, and $x$ the image of $y$ under the commuting maps in \eqref{eq:critical_diagram}. The same commuting diagrams confirm that this is well-defined, hence we have an isomorphism $\theta_{H,K}$ for each span $S$ and functor $F$. It is natural in $F$ since all of the maps in \autoref{eq:critical_diagram} commute with a map of functors $F \to F'$. It is also natural in $S$: along any equivariant map of spans $\sigma\colon S \to S'$,the values of $g$ and $\gamma$ do not change, hence $(\sigma(s),g,x)$ is sent to $(\gamma,\gamma^{-1}\sigma(s),y) = (\gamma,\sigma(i),y)$. The fact that this last expression comes from bringing $\sigma$ along the top route, then applying it to $(\gamma,i,y)$, is a bit of a diagram chase through \cite[Prop 4.10]{CaryMona}.
	%\cmar{checked! it sends every summand to the corresponding summand along $\sigma$, by an identity map}
	
	We extend the definition of $\theta_{K,K}$ to the formal unit $1_{G/K} \in \GE(G/K,G/K)$ by noting that on the formal unit the square \eqref{theta_HK} strictly commutes, so we can just define $\theta_{K,K}$ to be the identity. This extended $\theta_{K,K}$ is clearly an isomorphism and natural in $F$. Naturality in $S$ reduces to showing $\theta_{K,K}$ commutes with the canonical isomorphism between the identity and with composition with the span
	\[ \xymatrix @R=1em @C=1em{
		& S = G/K \ar@{=}[ld] \ar@{=}[rd] & \\
		G/K && G/K. } \]
	In particular, we must show the following commutes, where the vertical maps are the two canonical isomorphisms that are used when extending multiplication in $\GE$ to the formal unit, and when extending the action of spans on retractive spaces to the formal unit, respectively.
	\[ \xymatrix @C=10em{
		G \times_K F(e,K) \ar[d]_-\cong \ar@{=}[r]^-{\theta_{K,K} \textup{ (on the formal unit)}} &
		G \times_K F(e,K) \ar[d]_-\cong \\
		G/K \times_{G/K} (G \times_K F(e,K)) \ar@{<->}[r]^-{\theta_{K,K} \textup{ (when }S = G/K)}
		& G \times_K (G/K * F)(e,K)
	} \]
	We compute the bottom map at the point $(gK,G,x) \in G/K \times G \times F(e,K)$ by taking $\gamma = g$, then $q(i) = K$, and $F(\gamma^{-1} g \to e,q(i))$ is an identity map so $y = x$. Along the canonical identifications we therefore get the identity of $G \times_K F(e,K)$.
	
	Now we let $\underline{\theta_{H,K}}$ denote the pasting of the isomorphism above and the Beck-Chevalley isomorphism from \eqref{2_cell_to_strictified}:
\[\xymatrix @R=2em @C=3em{
	G\sE(G/H,G/K) \times i\Lambda \underline{R(X)^{hK}_\delta}  \ar[d]_-\sim \ar[r]   \drtwocell<\omit>{<0> \ \ \ \   \textup{BC}}  & i\Lambda \underline{R(X)^{hH}_\delta} \ar[d]^-\sim \\%
	G\sE(G/H,G/K) \times i\Lambda R(X)^{hK}_\delta  \ar[d]_-{\id\times f_K} \ar[r]   \drtwocell<\omit>{<0> \ \ \ \   \theta_{H,K}}  & i\Lambda R(X)^{hH}_\delta \ar[d]^-{f_H} \\%
	G\sE(G/H,G/K) \times G\sE(G/K, X)  \arrow[r]& G\sE(G/H, X).
}\]
%	Recall that the underlined category in the top-left, each object is a functor $F\colon \tG \times G/L \to \Rd(X)$ and span $T$ from $G/K$ to $G/L$. The morphisms are defined so that the map to $ i\Lambda R(X)^{hK}_\delta$ taking $(T,F)$ to $T * F$ is an equivalence of categories. The action of spans on this category is defined on objects to take $(S,(T,F))$ to $(S*T,F)$, and on morphisms precisely so that the Beck-Chevalley map is a natural isomorphism.
	
	We check that $\underline{\theta_{H,K}}$ is natural in maps $X \to X'$ in the sense of \autoref{pseudo_is_natural}. This breaks into two conditions, one for $\theta_{H,K}$ and one for the Beck-Chevalley map, which are checked at each object in the source. For $\theta_{H,K}$ the condition holds because the definition of $\theta_{H,K}$ ignores the projection to $X$, so it commutes with composing that projection with a map $X \to X'$. The condition for the Beck-Chevalley map is checked at each object of $G\sE(G/H,G/K) \times i\Lambda \underline{R(X)^{hK}_\delta}$, where it becomes the commutativity of \eqref{BC_commutes_with_functoriality}. In both cases if the span is the formal unit, the natural isomorphism is an equality, so the condition follows immediately.
	
	It remains to check the associativity and unit coherences from \autoref{pseudolinear}. The unit coherence is immediate because $\theta_{K,K}$ is an equality on the formal unit. For the associativity coherence, it suffices to check it once on every isomorphism class of objects, so we can discard the formal units from the categories $G\sE(G/H,G/K)$. The associativity coherence cube then breaks into two cubes, joined along one square face that has a Beck-Chevalley map. The top cube at the object $(S,T,(U,F))$ boils down to standard pasting results about Beck-Chevalley isomorphisms on the grid
	
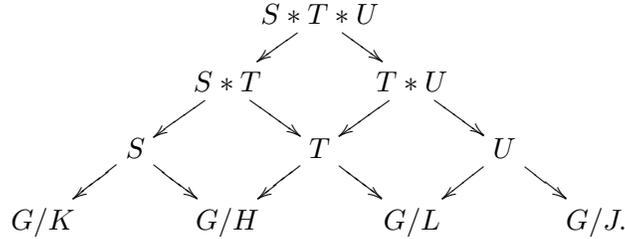
\[ \xymatrix @R=1em @!C=1em{
		&&& S * T * U \ar[ld] \ar[rd] && \\
		&& S * T \ar[ld] \ar[rd] && T * U \ar[ld] \ar[rd] & \\
		& S \ar[ld] \ar[rd] && T \ar[ld] \ar[rd] && U \ar[ld] \ar[rd] & \\
		G/K && G/H && G/L && G/J. } \]
	For the bottom cube, we need to show that the following two composites of 2-cells agree:
	\[\xymatrix@R=1.5em @!C=10em{
		& G\sE(G/L,G/K)\times i\Lambda R(X)^{hK}_\delta  \ar[rd]
		\\
		G\sE(G/L,G/H)\times G\sE(G/H,G/K)\times i\Lambda R(X)^{hK}_\delta \ar[dd]_-{\id\times f_K} \ar[rd] \ar[ru]  \xtwocell[rddd]{}<>{<0> \ \ \ \ \ \theta_{H,K}} \xtwocell[rr]{}<>{<0> \ \ \ \ \ \textup{BC}}
		&& i\Lambda R(X)^{hL}_\delta \ar[dd]^-{f_L}
		\\
		& G\sE(G/L,G/H)\times  i\Lambda R(X)^{hH}_\delta \ar[dd]^-{f_H}   \ar[ru]   \drtwocell<\omit>{<0> \ \ \ \ \ \theta_{L,H}}
		\\
		G\sE(G/L,G/H)\times G\sE(G/H,G/K)\times G\sE(G/K, X)  \arrow[rd]&& G\sE(G/L, X)
		\\
		&G\sE(G/L,G/H)\times  G\sE(G/H, X) \arrow[ru]
	}\]
	\[\xymatrix@R=1.5em @!C=10em{
		& G\sE(G/L,G/K)\times i\Lambda R(X)^{hK}_\delta  \ar[rd] \ar[dd]^-{\id\times f_K} \xtwocell[rddd]{}<>{<0> \ \ \ \ \ \theta_{L,K}}
		\\
		G\sE(G/L,G/H)\times G\sE(G/H,G/K)\times i\Lambda R(X)^{hK}_\delta \ar[dd]_-{\id\times f_K}\ar[ru]  \xtwocell[rd]{}<>{<0> \ \ \ \ =} 
		&& i\Lambda R(X)^{hL}_\delta \ar[dd]^-{f_L}
		\\
		& G\sE(G/L,G/K)\times G\sE(G/K,X) \ar[rd]
		\\
		G\sE(G/L,G/H)\times G\sE(G/H,G/K)\times G\sE(G/K, X) \ar[ru] \ar[rd] \xtwocell[rr]{}<>{<0> \ \ \ \ =}
		&& G\sE(G/L, X)
		\\
		&G\sE(G/L,G/H)\times  G\sE(G/H, X) \arrow[ru]
	}\]
%	\[\xymatrix@R=2em @C=4em{
%		G\sE(G/L,G/H)\times G\sE(G/H,G/K)\times i\Lambda R(X)^{hK}_\delta  \ar[d]_-{\id\times f_K} \ar[r]   \drtwocell<\omit>{<0> \ \ \ \   \theta_{H,K}}  & G\sE(G/L,G/H)\times  i\Lambda R(X)^{hH}_\delta \ar[d]^-{f_H}   \ar[r]   \drtwocell<\omit>{<0> \ \ \ \   \theta_{L,H}}  & i\Lambda R(X)^{hL}_\delta \ar[d]^-{f_L} \\%
%		G\sE(G/L,G/H)\times G\sE(G/H,G/K)\times G\sE(G/K, X)  \arrow[r]& G\sE(G/L,G/H)\times  G\sE(G/H, X) \arrow[r]& G\sE(G/L, X)
%	}\]
We check this at the triple $(T,S,F)$ where $S$ and $T$ are spans
\[ \xymatrix @R=1em @!C=1em{
	&& T * S \ar[ld] \ar[rd] & \\
	& T \ar[ld]_-s \ar[rd]^-r && S \ar[ld]_-p \ar[rd]^-q & \\
	G/L && G/H && G/K }
\]
and $\tG \times G/K \xrightarrow{F} \Rd(X)$ is a functor.
%
%\[\xymatrix @R=2em @C=3em{
%%		S_{K,H}\times G\sE(G/H, X)  \ar[d]_-{\id\times f_H} \ar[r]   \drtwocell<\omit>{<0> \ \ \ \   \theta_{K,H}}  & G\sE(G/K, X) \ar[d]^-{f_K} \\%
%%		S_{K,H}\times i\Lambda\underline\Rd_H(X)  \arrow[r]& i\Lambda\underline\Rd_K(X).
%		S_{L,K}\times i\Lambda\underline\Rd_K(X)  \ar[d]_-{\id\times f_K} \ar[r]   \drtwocell<\omit>{<0> \ \ \ \   \theta_{L,K}}  & i\Lambda\underline\Rd_L(X) \ar[d]^-{f_H} \\%
%		S_{H,K}\times G\sE(G/K, X)  \arrow[r]& G\sE(G/L, X)
%	}\]
%at value $(S\ast T, F)$.
The required commuting diagram is
\[ \xymatrix @R=1.5em{
	G \times_L (T * (S * F))(e,L) \ar@{<->}[d]^-{\theta_{L,H}} \ar@{<->}[r]^-{BC} &
	G \times_L ((T * S) * F)(e,L) \ar@{<->}[d]^-{\theta_{L,K}} \\
	T \times_{G/H} (G \times_H (S * F)(e,H)) \ar@{<->}[d]^-{T \times_{G/H} (\theta_{H,K})} &
	(T * S) \times_{G/K} (G \times_K F(e,K)) \ar@{=}[d] \\
	T \times_{G/H} (S \times_{G/K} (G \times_K F(e,K))) \ar@{=}[r] &
	(T \times_{G/H} S) \times_{G/K} (G \times_K F(e,K)).
} \]
We compare both routes from the lower-left to the top two terms, which expand out as
\[ T * (S * F) (e,L) = \coprod_{j \in s^{-1}(L)} (S * F)(e,r(j)) = \coprod_{j \in s^{-1}(L)} \coprod_{i \in p^{-1}(r(j))} F(e,q(i)) \]
\[ (T * S) * F (e,L) = \coprod_{(j,i) \in T * S, s(j) = L} F(e,q(i)). \]
%It isn't necessary to compute the $L$-action on this because just have to check that along this isomorphism, the formula for $\theta_{L,H}$ composed with $T \times_{G/H} \theta_{H,K}$ agrees with the formula for $\theta_{L,K}$.
and are identified in the obvious way. For simplicity, we can include further into the sum over all $j \in T$ and $i \in S$ when checking that the two maps agree. Along the left-hand route, $(b,a,g,x)$ goes to
\[ \begin{array}{rrrrlllll}
	(b, &a, &g, &x) &\in & T\times_{G/H} &(S\times_{G/K} &(G\times_K & F(e,K))) \\
	(b, &\gamma, &i = \gamma^{-1}a, &y) &\in & T\times_{G/H} &(G\times_H & \coprod_{i\in p^{-1}(H)} &F(e, q(i)))  \\
	(\rho, &j = \rho^{-1}b, &i = \rho^{-1} a, &z) &\in & G\times_L &(\coprod_{j\in s^{-1}(L)} &\coprod_{i\in p^{-1}(r(j))} &F(e, q(i))) \\
\end{array} \]
%\cmar{This was confusing to me before because the two values of $i$ you get are different.}%\mmar{yes, that also confusing to em when I was working through this, and once that clicked the formulas matches, hooray}
%\[\xymatrix @R=2em @C=.5em{
%	G\times_L(\coprod_{j\in s^{-1}(L)}\coprod_{i\in p^{-1}(r(j))}F(e, q(i)))
%	&\ni (\rho, j = \rho^{-1}b, i = \rho^{-1} a, z)  \ar@{<->}[d]^-{\theta_{L,H}}
%	\\
%	T\times_{G/H}(G\times_H \coprod_{i\in p^{-1}(H)}F(e, q(i)))
%	& \ni (b,\gamma, i = \gamma^{-1}a, y) \ar@{<->}[d]^-{\theta_{H,K}}
%	\\
%	T\times_{G/H}(S\times_{G/K}(G\times_K F(e,K)))
%	&\ni (b,a,g,x)
%}
%\]
Recall $\gamma$ is a representative of the $H$-coset $p(a)$ and $y$ is the image of $x$ under the map
$$F(e,K)\xrightarrow{F(e\to g^{-1}\gamma, K)} F(g^{-1}\gamma, K)=F(e, q(\gamma^{-1}a)),$$
which uses the equality $q(\gamma^{-1}a) = \gamma^{-1}gK$. Applying the formula again, $\rho$ is a representative of the $L$-coset $s(b)$, and we find the image of $(i = \gamma^{-1}a, y)$ under the map
$$(S*F)(e,H)\xrightarrow{(S*F)(e\to \gamma^{-1}\rho, H)} (S*F)(\gamma^{-1}\rho, H) = (S*F)(e, r(\rho^{-1}b))$$
which uses the equality $r(\rho^{-1}b) = \rho^{-1}\gamma H$. Expanding out this map gives the commuting diagram
\[ \resizebox{\textwidth}{!}{ \xymatrix @R=2em @C=1em{
	\coprod_{i\in p^{-1}(H)} F(e, \ q(i)) \ar[r]^-{} &
	\coprod_{i\in p^{-1}(H)} F(\gamma^{-1}\rho, \ q(i)) \ar@{<->}[r]^-\cong &
	\gamma^{-1}\rho(\coprod_{i\in p^{-1}(\rho^{-1}\gamma H)}F(e,\ q(i))) \ar@{=}[r] &
	\coprod_{i\in p^{-1}(r(\rho^{-1}b))}F(e,\ q(i))) \\
	F(e,\ q(\gamma^{-1}a) \ar[r] \ar[u] &
	F(\gamma^{-1}\rho,\ q(\gamma^{-1}a)) \ar@{=}[r] \ar[u] & 
	\gamma^{-1}\rho F(e,\ q(\rho^{-1}a)) \ar@{=}[r] \ar[u] & 
	F(e,\ q(\rho^{-1}a)) \ar[u]
} } \]
where the horizontal maps on the left apply $F(e \to \gamma^{-1}\rho,\ q(i))$, the middle square commutes by the definition of the top isomorphism, and the square on the right commutes because $\gamma^{-1}\rho$ acts trivially on all morphisms in $R(X)$. This gives $(i = \rho^{-1}a, z)$ as the last two coordinates of the answer, where $z \in F(e, \ q(\rho^{-1}a))$ is the image of $y$ along the bottom row.

The other route takes $(b,a,g,x)$ to a possibly different value of $(\rho,j,i,z)$. In this case $\rho$ is computed as a representative of the $L$-coset in the image of $(b,a) \in T * S$, but this is the same as $s(b)$, so we can take the same value for $\rho$. Then the point $(j,i) \in T * S$ is computed as the image of $(b,a)$ under $\rho^{-1}$, so we again get $(\rho^{-1}b,\rho^{-1}a)$, as above. Finally, $z$ in this case is the image of $x$ under the map
$$F(e,K)\xrightarrow{F(e\to g^{-1}\rho, K)} F(g^{-1}\rho, K)=F(e, q(\rho^{-1}a)).$$
%Let $a$ and $b$ be the projections of $T\times_{G/H} S$ to $T$ and $S$, respectively. On the other hand we have
%\[\xymatrix @R=2em @C=.2em{
%&( T\times_{G/H} S)\times_{G/K}(G\times_K F(e,K)))  & \ni (t,s,g,x)\ar[d]\\
%G\times_L(sa)_!(qb)^*F(e,L) \ar@{=}[r] & \coprod_{(j,i)\in (sa)^{-1}(L)} F(e, (qb)(j,i)) &\ni (\rho, (\rho^{-1}t, \rho^{-1}s), z)
%}\] 
%where we can pick $\rho$ to be the same representative of the coset $s(t)=(sa)(t,s)$ as before since we have shown the map is independent of coset representative choice, and we note also that $(qb)(j,i)=q(i)$. Here the element  $z$ is defined as the image of $x$ under the map
%
%We need to show that the two formulas for $z$ agree. Note that we already have that in both formulas $z$ lives in component $F(e, q(\rho^{-1}s))$.
This agrees with the previous definition of $z$ since the following diagram commutes.

\[ \resizebox{\textwidth}{!}{\xymatrix @R=2em @C=10em{
& x \in F(e,\ K)\ar@{=}[d]\ar[r]^-{F(e\to g^{-1}\rho, K)} \ar[lddd]_-{F(e\to g^{-1}\gamma, K)} & F(g^{-1}\rho, K)\ar@{=}[d]\\
& F(\rho^{-1}g, \ q(\rho^{-1}a))\ar@{=}[d] \ar[r]^-{F(\rho^{-1}g\to e, \ q(\rho^{-1}a))} & F(e, \ q(\rho^{-1}a))\ar@{=}[d] \ni z\\
& F(\gamma^{-1}g, \ q(\gamma^{-1} a))\ar[r]^-{F(\gamma^{-1}g\to \gamma^{-1}\rho, \ q(\gamma^{-1}a))} \ar[d]_-{F(\gamma^{-1}g\to e,\  q(\gamma^{-1}a))}&F(\gamma^{-1}\rho,\  q(\gamma^{-1}a))\\
F(g^{-1}\gamma, \ K) \ar@{=}[r] & F(e, \ q(\gamma^{-1}a)) \ni y \ar[ur]_-{\ \ \ \ \ \ \ \ F(e\to \gamma^{-1}\rho,\  q(\gamma^{-1}a))}&
}}\]
\end{proof}

\section{The fiber map of the map $\Sigma_G^\infty X_+\to \bA_G(X)$ }\label{fiber}

In \autoref{map} we have constructed a map $\Sigma_G^\infty X_+\to \bA_G(X)$ as a map of spectral Mackey functors ($\GB^\Wald$-modules). Considering them now as $G$-spectra, the genuine $G$-fixed points of both source and target satisfy tom Dieck splittings. By \cite[Thm 1.2.]{BD}, this map is compatible with the splitting on fixed points in the sense that there is a commuting diagram in the homotopy category %\cmar{I think this is all we need}\mmar{agreed with this}
\begin{equation}\label{compatibility}
\xymatrix{ (\Sigma_G^\infty X_+)^G \ar[r]^-{\simeq} \ar[d] &\prod_{(H) \leq G}  \Sigma^\infty (X^H_{hWH})_+ \ar[d]\\
 \bA_G(X)^G\ar[r]^-\simeq  & \prod_{(H) \leq G} \bA (X^H_{hWH})
 }\end{equation}
and a similar commuting diagram on the $H$-fixed points for each subgroup $H\leq G$. Each of the vertical maps on the right is the usual inclusion of stable homotopy into nonequivariant $A$-theory. We note that the model that we use for the fixed points of the suspension $G$-spectrum is precisely that from \cite{BD}, which we need for this compatibility to work.
%\mnote{yes, or something like that: the right vertical map is the usual one, but the left hand one is what they construct in the paper, and it is the $H$-levels of what we build, and we need that so that the compatibility works}

%We will now deduce the following theorem.

Diagram \eqref{compatibility} implies that the fiber of the $(\Sigma_G^\infty X_+)^G \to \bA_G(X)^G$ splits over conjugacy classes of subgroups. We next apply the nonequivariant stable parametrized $h$-cobordism theorem to identify the terms in the splitting.
\begin{df}
	If $M$ is a compact smooth manifold, let $\mc H(M)$ denote the space of $h$-cobordisms on $M$. Let $\mc H^\infty(M)$ be the homotopy colimit of the spaces $\mc H(M \times I^n)$ along the stabilization maps from \cite{igusa}.\footnote{A new treatment of the definition of the stable $h$-cobordism space as an $(\infty,1)$ functor is given in \cite{pieper}. We revisit this definition and give a treatment of stable $h$-cobordism spaces for manifolds with corners in upcoming work with Goodwillie and Igusa on spaces of equivariant $h$-cobordisms.}
\end{df}

\begin{thm}[Stable parametrized $h$-cobordism theorem, \cite{waldhausennew}]\label{parametrized_h_cobordism_theorem}
	For a smooth compact manifold $M$, there is a natural weak equivalence between $\mc H^\infty(M)$ and the homotopy fiber of $\Omega^\infty\Sigma^\infty_+ M \to \Omega^\infty\bA(M)$.
\end{thm}

Note that $\bA(M)$ can be defined using either homotopy finite or homotopy finitely dominated spaces in this theorem, because the difference only changes $\pi_{-1}$ of the homotopy fiber, and is therefore invisible once we take infinite loop spaces.
%As a result, the homotopy type of $\mc H^\infty(M)$ depends only on the homotopy type of $M$, so $\mc H^\infty(X)$ can be defined for any finite CW complex $X$ by taking a compact smooth manifold homotopy equivalent to it.

We cannot apply \autoref{parametrized_h_cobordism_theorem} right away because the homotopy orbits $M^H_{hWH}$ are not a compact smooth manifold. However, it is standard to extend the definition of $\mc H^\infty$ to infinite CW complexes by writing them as a filtered homotopy colimit of $\mc H^\infty$ of the finite subcomplexes, thickened into compact manifolds. This can be performed so that \autoref{parametrized_h_cobordism_theorem} holds with $M$ replaced by any CW complex $X$. We will describe this in detail only in the case when $X = M^H_{hWH}$; in this case we conclude in \autoref{h_infinity_colimit} that we can form  the filtered colimit using representation discs. We start with a definition and proposition necessary to make this construction explicit.

%because in this case we can form the filtered colimit using representation discs \autoref{compact_expression}).

%This becomes a precise definition if we know how $\mc H$ or $\mc H^\infty$ varies along smooth embeddings of compact manifolds. A treatment of this can be found in the unpublished thesis of Malte Pieper \cite{pieper}, which we summarize in the following theorem.%A treatment on how to carefully define a stable space of equivariant $h$-cobordisms will appear in forthcoming work with Goodwillie and Igusa.
%
%\begin{thm}[\cite{pieper}]\label{functorial_h_cobordism}
%	The stable $h$-cobordism space is an ($\infty$,1) functor from smooth compact manifolds and embeddings to spaces. The equivalence of \autoref{parametrized_h_cobordism_theorem} is an equivalence of ($\infty$,1) functors.
%\end{thm}
%
%\cmar{removed a paragraph below}
%Thus our last task is to write $M^H_{hWH}$ as a sequential colimit of compact manifolds, so that we can interpret the fiber of
%\[ \Omega^\infty\Sigma^\infty_+ M^H_{hWH} \to \Omega^\infty\bA(M^H_{hWH}) \]
%as a homotopy colimit of actual spaces of $h$-cobordisms. \autoref{colimit_prop} allows us to express $M^H_{hWH}$ as a colimit of compact manifolds, but its statement works in general for any $G$-space $X$. 

\begin{df}
Let $X$ be a $G$-space. Define $X_H$ to be the subspace of $X$ consisting of points with isotropy group exactly $H$:
$$X_H=\{x\in H \ |\  G_x=H\}= X^H\backslash \bigcup_{K>H} X^K.$$
If $X$ is a smooth compact $G$-manifold, instead of removing $X^K$, we remove an open tubular neighborhood of $X^K$. This produces a homotopy-equivalent version of $X_H$ that is a smooth compact $G$-manifold with corners.
\end{df}

\begin{rem}\label{colim}
In what follows we will be taking colimits over representations $V$, denoted $\underset{V} \colim$. These colimits are taken over inclusions of finite dimensional representations $V$ of $G$, inside a fixed complete $G$-universe. They contain a cofinal subsystem that is sequential (multiples of the regular representation), and the maps are closed inclusions, so these are also  homotopy colimits.
\end{rem}

\begin{prop}\label{colimit_prop}
Suppose $X$ is a $G$-space. Then for each subgroup $H$, the collapse maps $V \to *$  induce an equivalence 
$$\alpha\colon \underset{V}\colim(X\times V)_H \longrightarrow X^H$$
\end{prop}
%Here the colimit is taken over inclusions of finite dimensional representations $V$ of $G$, inside a fixed complete $G$-universe. It contains a cofinal subsystem that is sequential (multiples of the regular representation), and the maps are closed inclusions, so it is also a homotopy colimit.

Intuitively, \autoref{colimit_prop} is true because as we multiply $X$ by higher and  higher-dimensional representation discs, the codimension of the subspaces $X^K \subseteq X^H$ goes to infinity.

\begin{proof}
%\mona{Since $V\simeq D(V)$ are equivariantly equivalent, we run the proof with $V$ instead of $D(V)$ so that we are not confined to unit disks.} 
We show that the relative homotopy groups $\pi_n\big(X^H, \underset{V}\colim(X\times V)_H\big)$ are trivial. Suppose that we have a diagram
%\begin{center} 
$$\xymatrix @C=3em @R=3em{
S^{n-1} \ar[r]^-\gamma \ar@{^{(}->}[d]_-i  & \underset{V}\colim(X\times V)_H \ar[d]^-\alpha \\
D^n \ar[r]^-\beta \ar@{-->}[ru]^-{\tilde{\gamma}?} & X^H
}$$
%\end{center}
Since $S^{n-1}$ is compact, $\gamma$ factors through some stage $V$ of the colimit,
$$\gamma\colon S^{n-1} \longrightarrow  (X\times V)_H.$$ Let $\rho$ be the regular representation of $G$.  Denote by $p_1$ and $p_2$ the projections to the $X$ and the $V$ components.
%Note that for any $y\in S^{n-1}$, by the assumed commutativity of the diagram, $(p_1 \circ \gamma)(y)\in X^H$. Therefore $(p_2 \circ \gamma) (y)\in D(V)_H$.
We will define a lift $\tilde{\gamma}: D^n\to (X\times (V\oplus \rho))_H.$

Let $y_t$ be a point in $D^n$, where $t$ indicates its position on the radius of $D^n$ parametrized by the unit interval, so that $y_0$ is on the boundary $S^{n-1}$ and  $y_1$ is the center of $D^n$.

Define the $X$-component $p_1 \circ \tilde{\gamma}(y_t)$ to be $\beta(y_t)\in X^H$ for any $y_t\in D^n$. Now we define the $(V\oplus \rho)$-component $p_1 \circ \tilde{\gamma}(y_t)$ so that  when $t > 0$ it lands in $(V\oplus \rho)_H$, namely so that it is $H$-fixed but not $K$-fixed for any $K>H$.   Let $\{\chi_g\}$ be a basis of the regular representation $\rho$ of $G$, and note that $ \sum\limits_{h\in H} \chi_h$ has isotropy exactly $H$. We define the map
$$\tilde{\gamma} \colon D^n\longrightarrow V\oplus \rho$$ by

\[  y_t \mapsto \begin{cases}
\big((p_2 \circ \gamma)(y_0), \ 2t \sum\limits_{h\in H} \chi_h\big)&\text{   for }t\leq \frac{1}{2} \\
\big((2-2t)(p_2 \circ \gamma)(y_0),\  \sum\limits_{h\in H} \chi_h\big)&\text{   for }t\geq \frac{1}{2}
\end{cases} \]
This assignment is easily seen to be continuous. When $t = 0$ it lands in $ (X\times (V\oplus \rho))_H$ because $\gamma$ does. When $t > 0$ it lands in $ (X\times (V\oplus \rho))_H$ because the second coordinate is in $(V\oplus \rho)_H$.
\end{proof}

Now suppose that $M$ is a smooth compact manifold. We will use \autoref{colimit_prop} to express $M^H_{hWH}$ as a colimit of compact submanifolds.
\begin{cor}\label{h_infinity_colimit} Let $M$ be a smooth compact $G$-manifold, and let $H$ be a subgroup of $G$ and $WH$ the Weyl group $NH/H$. Then there is an equivalence
	$$\mc H^\infty(M^H_{hWH}) \simeq \underset{V}\colim \mc H^\infty(M_V),$$ where $M_V=(M\times D(V))_H)/WH$ are compact smooth manifolds.
\end{cor}
\begin{proof}
Since $V\simeq D(V)$ is a $G$-equivariant equivalence, by \autoref{colimit_prop}, there is an equivalence 
$$\alpha\colon \underset{V}\colim(M\times D(V))_H \longrightarrow M^H,$$ which is $WH$-equivariant. Thus it induces equivalences
\begin{eqnarray*}
M^H_{hWH} & \simeq & \big(\underset{V}\colim (M\times D(V))_H \big)_{hWH}\\
&\simeq  & \underset{V}\colim \big( ((M\times D(V))_H)_{hWH}\big) \\
&\simeq & \underset{V}\colim \big( ((M\times D(V))_H)/WH\big)
\end{eqnarray*}
Lastly, note that the manifolds $M_V=(M\times D(V))_H)/WH$ are compact when $M$ is so.
\end{proof}

%\begin{cor}\label{compact_expression}
%Suppose $X$ is a $G$-space, and let $H$ be a subgroup of $G$ and $WH$ the Weyl group $NH/H$.  There is an equivalence $$X^H_{hWH}\simeq  \underset{V}\colim \big( ((X\times D(V))_H)/WH\big)$$.
%\end{cor}
%
%\begin{proof}
%The map $\alpha$ is $WH$-equivariant, so the equivalence $\alpha$ induces equivalences 
%
%\begin{eqnarray*}
%X^H_{hWH} & \simeq & \big(\underset{V}\colim (X\times D(V))_H \big)_{hWH}\\
%&\simeq  & \underset{V}\colim \big( ((X\times D(V))_H)_{hWH}\big) \\
%&\simeq & \underset{V}\colim \big( ((X\times D(V))_H)/WH\big)
%\end{eqnarray*}
%\end{proof}
%
%
%\begin{notn}\label{MV_notn}
%To make the notation less unwieldy, for a fixed $H$, we will denote the space $((M\times D(V))_H)/WH$ by $M_V$. Note that if $M$ is a smooth compact $G$-manifold, the spaces $M_V$ are smooth compact manifolds as well.
%\end{notn}
%
%%The following is then true by the definition of $\mc H^\infty$ for infinite complexes.
%
%\begin{cor}\label{h_infinity_colimit}
%	$\mc H^\infty(M^H_{hWH}) \simeq \underset{V}\colim \mc H^\infty(M_V)$.
%\end{cor}

It remains to describe the behavior of $A$-theory on filtered colimits. We say that a filtered diagram of closed inclusions of spaces $\{X_\alpha\}$, with colimit $X = \bigcup_\alpha X_\alpha$, is \emph{strongly filtered} if any map from a finite CW complex $K$ to $X$ factors through some $X_\alpha$. In particular, sequential colimits along closed inclusions are strongly filtered.
%\mmar{references for these props? I remember Inna told us once a reference when you were visiting her. If we don't find references, we probably need to say something more}\cmar{I just gave proofs. I can't do the following one for finitely dominated though, only finite.}
%\mmar{that's ok, we only need hf, right? :)}\cmar{Yep, it's good enough for government work}

We start by proving the desired commutation with colimits for the category of strictly finite relative cell-complexes, which we denote by $R^f(X)$.
\begin{prop}\label{prop1} 
If $\{X_\alpha\}$ is a strongly filtered system of spaces and $X =  \colim X_\alpha$, then
$$\underset{\alpha}\colim R^f(X_\alpha) \longrightarrow R^f(X)$$
is an equivalence of categories, where $R^f$ denotes the category of retractive spaces that are isomorphic to finite cell complexes relative to the base.
\end{prop}

\begin{proof}
	Let $Y$ be a retractive space over $X$ such that $X \to Y$ is a finite relative cell complex. The closure $\overline{Y\setminus X}$ is covered by the union of finitely many closed cells, each of which has projection landing in some $X_\alpha$, so the entire closure has projection landing in some $X_\alpha$. Then $Y$ is the pushout in spaces of $X$ and $Y \setminus (X \setminus X_\alpha)$ along $X_\alpha$, and therefore $Y$ comes from $R^f(X_\alpha)$. Therefore the map from $\underset{\alpha}\colim R^f(X_\alpha)$ is essentially surjective. A similar argument shows it is also full, and it is faithful because a morphism of spaces over $X_\alpha$ can be recovered from its pushout to $X$ simply by deleting $X \setminus X_\alpha$.
%	
%	This subcategory is equivalent to the entire category because any homotopy finitely dominated space $Y$ can be replaced by a relative CW complex $X \to Y$, and then it admits a map to a fibrant replacement of a relative finite CW complex $X \to Z$ over $X$. As a result, there is a map $Y \to Z$ that commutes with the projection to $X$ up to homotopy. Thickening $Y$ up to a cylinder and then including the other end, we get a zig-zag of equivalences between $Y$ and another retractive space whose projection to $X$ factors through 
\end{proof}

\begin{prop}\label{prop2}
If $\sC = \colim \sC_\alpha$ is a filtered colimit of Waldhausen categories, then the canonical map
$$\hocolim K(\sC_\alpha) \longrightarrow K(\sC)$$
is an equivalence of spectra.
\end{prop}

\begin{proof}
	First note that a filtered colimit of Waldhausen categories is again Waldhausen. The proof is then is another exercise in filtered colimits, using the fact that they commute with finite products, that a filtered colimit of simplicial sets is always a homotopy colimit, and that homotopy colimits of spectra are computed by taking the homotopy colimit of each spectrum level. See also \cite[page 20, (9)]{quillen1} for the analogous result for exact categories.
\end{proof}

\begin{cor}\label{prop3}
	If $\{X_\alpha\}$ is a strongly filtered system of spaces and $X =  \colim X_\alpha$, then
	$$\underset{\alpha}\hocolim K(\Rhf(X_\alpha)) \longrightarrow K(\Rhf(X))$$
	is an equivalence of spectra.
\end{cor}

Recall that $\Rhf(X)$ are homotopy finite relative complexes, i.e. relative cell complexes which are homotopy equivalent to finite ones. 

\begin{proof}
By \autoref{prop1} and \autoref{prop2}, the corresponding map on $K$-theory of finite complexes is an equivalence:
$$\underset{\alpha}\hocolim K(R^f(X_\alpha)) \simar K(R^f(X)).$$
The corollary follows from the standard fact that the inclusion $R^f(X)\to \Rhf(X)$ satisfies Waldhausen's approximation theorem. In particular, any map $Y \to Z$ of spaces that are finite complexes relative to $X$, can be factored into a relative finite complex $Y \to Y'$ followed by a weak equivalence $Y' \overset\sim\to Z$.
\end{proof}

%\mmar{double check this. If something is homotopy finite, its wall finiteness obstruction is 0...how was this coming in to show that with finitely many cells you can make it actually finite? }
%\mona{Even though we do not need this here since we are using the definition of $A(X)$-theory as $K$-theory of $R^f(X)$, we note that we could draw the same conclusion about $K(R_{hf}(X))$ since map induced by the inclusion of $R^f(X)$ to $R_{hf}(X)$ gives an isomorphism on higher homotopy groups by a cofinality argument, and the commutation with the colimit can be seen to be an isomorphism on $\pi_0$ by a direct computation.}

We now have all the ingredients we need to prove an equivariant stable parametrized $h$-cobordism theorem.

\begin{df}\label{HGdef}
	Let $M$ be a compact smooth $G$-manifold. Define $\textbf{H}_G(M)$ as the fiber of the map  $\Sigma_G^\infty M_+\to \bA_G(M)$  constructed in \autoref{map}.
\end{df}

\begin{thm}\label{fiber_identification}
	Let $M$ be a compact smooth $G$-manifold. On the infinite loop space level, there is an equivalence
	$$\Omega^\infty  \textbf{H}_G(M)^G\simeq \prod_{(H) \leq G} \mc H^\infty(M^H_{hWH}),$$
	where $\mc H^\infty(X)$ is the stable $h$-cobordism space of $X$. An analogous result holds for $H$-fixed points for subgroups $H$ of $G$.
\end{thm}
%Now we combine these result to prove \autoref{fiber_identification}.

\begin{proof}%[Proof of \autoref{fiber_identification}.]
%Let $\mc H_G(M)$ be the fiber of the map  $\Sigma_G^\infty M_+\to \bA_G(M)$ constructed in \autoref{map}.
Since fibers commute with fixed points and since the fixed points of the map $\Sigma_G^\infty M_+\to \bA_G(M)$ satisfy the compatibility from \autoref{compatibility}, $\textbf{H}_G(M)^G$ splits over conjugacy classes of subgroups $(H)$, and we can identify the $(H)$ component as the fiber of $\Sigma^\infty (M^H_{hWH})_+\to \bA(M^H_{hWH}).$ Moreover, since the infinite loop space functor commutes with products, we can identify $\Omega^\infty \textbf{H}_G(M)^G$ with the product of the infinite loop spaces of the fibers of the $(H)$ components of the map.

Recall \autoref{h_infinity_colimit} and \autoref{prop3}, we can express the map $\Sigma^\infty (M^H_{hWH})_+\to \bA(M^H_{hWH})$ as map of homotopy colimits over representations $V$
$$\hocolim \Sigma^\infty_+ M^H_{hWH} \longrightarrow \hocolim \bA(M^H_{hWH})$$ %\mmar{is the suspension spectrum also such hocolim. it's surely a colim since it's a left adjoint}
By \autoref{parametrized_h_cobordism_theorem},
%the nonequivariant parametrized stable $h$-cobordism theorem \cite{waldhausennew}, 
since all $M^H_{hWH}$ are smooth  compact manifolds, and since homotopy fibers in spectra commute with homotopy colimits, we can identify  the fiber of this map with $\hocolim  \textbf{H}(M^H_{hWH})$, where $ \textbf{H}(M^H_{hWH})$ is the stable $h$-cobordism spectrum whose infinite loop space is the stable $h$-cobordism space $\mc H^\infty(M^H_{hWH})$. 

Therefore, we can conclude that 
    $$\Omega^\infty \textbf{H}_G(M)^G\simeq \prod_{(H) \leq G} \Omega^\infty \hocolim \textbf{H}(M^H_{hWH}),$$
    Commuting the infinite loop space functor $\Omega^\infty$ with the filtered homotopy colimit and applying \autoref{h_infinity_colimit}, %and since $\mc H^\infty(M^H_{hWH})= \hocolim \mc H^\infty(M_V)$ by \autoref{h_infinity_colimit},
    we get the desired equivalence 
  $$\Omega^\infty \textbf{H}_G(M)^G\simeq \prod_{(H) \leq G} \mc H^\infty(M^H_{hWH}).$$
\end{proof}

\begin{rem}
	This version of the theorem relates the categorically-constructed space $\Omega^\infty \textbf{H}_G(M)^G$ to a product of spaces of nonequivariant $h$-cobordisms.
%	is defined as the fiber of the map  $\Sigma_G^\infty M_+\to \bA_G(M)$ and not in terms of equivariant $h$-cobordisms a priori.
	In upcoming joint work with Goodwillie and Igusa we define a single stable space of \emph{equivariant} $h$-cobordisms that splits into the same product, and using \autoref{fiber_identification} conclude that it is equivalent to $\Omega^\infty \textbf{H}_G(M)^G$.
\end{rem}

  \bibliographystyle{amsalpha}
  \bibliography{references}

\begingroup%
\setlength{\parskip}{\storeparskip}% Restore \parskip within this scope

\end{document}